\documentclass[12pt,reqno]{amsart}

\newcommand{\equi}{\Longleftrightarrow}
\usepackage{times}
\usepackage[T1]{fontenc}
\usepackage{mathrsfs}
\usepackage{latexsym}
\usepackage[dvips]{graphics}
\usepackage[dvips]{graphicx}
\usepackage{epsfig}
\usepackage{amsmath,amsfonts,amsthm,amssymb,amscd}
\input amssym.def
\input amssym.tex
\usepackage{color}
\usepackage{nicefrac}

\newcommand{\comment}[1]{}

\newcommand{\f}{\mathcal{F}}
\newcommand{\pp}{\ensuremath{\mathbb{P}}}

\newcommand\be{\begin{equation}}
\newcommand\ee{\end{equation}}
\newcommand\bea{\begin{eqnarray}}
\newcommand\eea{\end{eqnarray}}
\newcommand\nbea{\begin{eqnarray*}}
\newcommand\neea{\end{eqnarray*}}
\newcommand\bi{\begin{itemize}}
\newcommand\ei{\end{itemize}}
\newcommand\ben{\begin{enumerate}}
\newcommand\een{\end{enumerate}}

\newcommand{\ncr}[2]{\left({#1 \atop #2}\right)}

\newtheorem{thm}{Theorem}[section]

\newtheorem{lem}[thm]{Lemma}

\newtheorem{defi}[thm]{Definition}

\newcommand{\twocase}[5]{#1 \begin{cases} #2 & \text{{\rm #3}}\\ #4
&\text{{\rm #5}} \end{cases}   }

\newcommand{\E}{\ensuremath{\mathbb{E}}}
\newcommand{\R}{\ensuremath{\mathbb{R}}}
\newcommand{\C}{\ensuremath{\mathbb{C}}}
\newcommand{\Z}{\ensuremath{\mathbb{Z}}}

\newcommand{\N}{\mathbb{N}}

\newcommand{\foh}{\frac{1}{2}}

\newcommand{\ga}{\alpha}

\newcommand{\gep}{\epsilon}

\numberwithin{equation}{section}

\newcommand{\p}{\prime}
\newcommand{\s}{\sigma}
\newcommand{\ds}{\displaystyle}

\newcommand{\Mod}[1]{\,(\operatorname{mod}\, #1)}
\newcommand{\con}{\equiv}

\DeclareMathOperator{\Tr}{Tr}
\DeclareMathOperator{\M}{M}

\newcommand{\bal}{\begin{align}}
\newcommand{\eal}{\end{align}}

\newcommand{\nn}{\nonumber}

\newcommand{\surj}{\twoheadrightarrow}
\newcommand{\inj}{\hookrightarrow}

\begin{document}

\title[Limiting Spectral Measure for Symmetric Block Circulant Matrices]{The Limiting Spectral Measure for Ensembles of Symmetric Block Circulant Matrices}

\author{Murat Kolo$\breve{{\rm g}}$lu}\email{Murat.Kologlu@williams.edu}
\address{Department of Mathematics and Statistics, Williams College,
Williamstown, MA 01267}

\author{Gene S. Kopp}\email{gkopp@uchicago.edu}
\address{Department of Mathematics, University of Chicago, Chicago, IL 60637}

\author{Steven J. Miller}\email{Steven.J.Miller@williams.edu}
\address{Department of Mathematics and Statistics, Williams College,
Williamstown, MA 01267}

\subjclass[2010]{15B52, 60F05, 11D45 (primary), 60F15, 60G57, 62E20 (secondary)}

\keywords{limiting spectral measure, circulant and Toeplitz matrices, random matrix theory, convergence, method of moments, orientable surfaces, Euler characteristic}

\date{\today}

\thanks{The first and second named authors were partially supported by Williams College and NSF grants DMS0855257 and DMS0850577, and the third named author was partly supported by NSF grant DMS0970067. It is a pleasure to thank our colleagues from the Williams College 2010 SMALL REU program as well as the participants of the ICM Satellite Meeting in Probability \& Stochastic Processes (Bangalore, 2010) for many helpful conversations, especially Arup Bose and Rajat Hazra. We would also like to thank Elizabeth Townsend Beazley for comments on Wentao Xiong's senior thesis, which is the basis of Appendix \ref{sec:appwentao}}

\begin{abstract} Given an ensemble of $N \times N$ random matrices, a natural question to ask is whether or not the empirical spectral measures of typical matrices converge to a limiting spectral measure as $N\to\infty$. While this has been proved for many thin patterned ensembles sitting inside all real symmetric matrices, frequently there is no nice closed form expression for the limiting measure. Further, current theorems provide few pictures of transitions between ensembles. We consider the ensemble of symmetric $m$-block circulant matrices with entries i.i.d.r.v. These matrices have toroidal diagonals periodic of period $m$.  We view $m$ as a ``dial'' we can ``turn'' from the thin ensemble of symmetric circulant matrices, whose limiting eigenvalue density is a Gaussian, to all real symmetric matrices, whose limiting eigenvalue density is a semi-circle.  The limiting eigenvalue densities $f_m$ show a visually stunning convergence to the semi-circle as $m \to \infty$, which we prove.

In contrast to most studies of patterned matrix ensembles, our paper gives explicit closed form expressions for the densities.  We prove that $f_m$ is the product of a Gaussian and a certain even polynomial of degree $2m-2$; the formula is the same as that for the $m \times m$ Gaussian Unitary Ensemble (GUE).  The proof is by derivation of the moments from the eigenvalue trace formula.  The new feature, which allows us to obtain closed form expressions, is converting the central combinatorial problem in the moment calculation into an equivalent counting problem in algebraic topology. We end with a generalization of the $m$-block circulant pattern, dropping the assumption that the $m$ random variables be distinct. We prove that the limiting spectral distribution exists and is determined by the pattern of the independent elements within an $m$-period, depending on not only the frequency at which each element appears, but also the way the elements are arranged. \end{abstract}

\maketitle

\tableofcontents

%%%%%%%%%%%%%%%%%%%%%%%%%%%%%%%%%%%%%%%%%%%%%%%%%%%%%%%%%%%%%%%%%%%%%%%%%%%%%%%%%%%%%%%%%%%%%%%%%%%%%%%%%%%%
%%%%%%%%%%%%%%%%%%%%%%%%%%%%%%%%%%%%%%%%%%%%%%%%%%%%%%%%%%%%%%%%%%%%%%%%%%%%%%%%%%%%%%%%%%%%%%%%%%%%%%%%%%%%

\section{Introduction}

\subsection{History and Ensembles}

Random matrix theory is the study of properties of matrices chosen according to some notion of randomness, which can range from taking the structurally independent entries as independent identically distributed random variables to looking at subgroups of the classical compact groups under Haar measure. While the origins of the subject go back to Wishart's \cite{Wis} investigations in statistics in the 1920s, it was Wigner's work \cite{Wig1, Wig2, Wig3, Wig4, Wig5} in the 1950s and Dyson's \cite{Dy1, Dy2} a few years later that showed its incredible power and utility, as random matrix ensembles successfully modeled the difficult problem of the distribution of energy levels of heavy nuclei. The next milestone was twenty years later, when Montgomery and Dyson \cite{Mon} observed that the behavior of eigenvalues in certain random matrix ensembles correctly describe the statistical behavior of the zeros of the Riemann zeta function. The subject continues to grow, with new applications ranging from chemistry to network theory \cite{MNS} to transportation systems \cite{BBDS,KrSe}. See \cite{FM,Hay} for a history of the development of the subject and the discovery of some of these connections.

One of the most studied matrix ensembles is the ensemble of $N \times N$ real symmetric matrices.  The $N$ entries on the main diagonal and the $\frac{1}{2}N(N-1)$ entries in the upper right are taken to be independent, identically distributed random variables from a fixed probability distribution with density $p$ having mean $0$, variance $1$, and finite higher moments.  The remaining entries are filled in so that the matrix is real symmetric. Thus \bea \mbox{Prob}(A) \ = \ \prod_{1 \le i \le j \le N}
p(a_{ij}), \ \ \ \text{Prob}\left(A: a_{ij} \in [\alpha_{ij},
\beta_{ij}]\right) \ = \ \prod_{1 \le i \le j \le N}
\int_{x_{ij}=\alpha_{ij}}^{\beta_{ij}} p(x_{ij}) dx_{ij}. \eea We want to understand the eigenvalues of $A$ as we average over the family. Let $\delta(x - x_0)$ denote the shifted Delta functional (i.e., a unit point mass at $x_0$, satisfying $\int f(x) \delta(x-x_0)dx = f(x_0)$). To each $A$ we associate its empirical spacing measure:
\bea
\mu_{A,N}(x) & \ = \ & \frac{1}{N} \sum_{i=1}^N \delta\left( x
- \frac{\lambda_i(A)}{\sqrt{N}} \right).
\eea
Using the Central Limit Theorem, one readily sees that the correct scale to study the eigenvalues is on the order of $\sqrt{N}$.\footnote{\label{foot:sizeevaluesfull}$\sum_{i=1}^N \lambda_i^2 = {\rm Trace}(A^2) = \sum_{i,j \le N} a_{ij}^2$; as the mean is zero and the variance is one for each $a_{ij}$, this sum is of the order $N^2$, implying the average square of an eigenvalue is $N$.} The most natural question to ask is: How many normalized eigenvalues of a `typical' matrix lie in a fixed interval as $N\to\infty$? Wigner proved that the answer is the semi-circle. This means that as $N\to\infty$ the empirical spacing measures of almost all $A$ converge to the density of the semi-ellipse (with our normalization), whose density is \be\label{eq:densityfwigforwigner} \twocase{f_{\rm \text{{\rm Wig}}}(x) \ = \ }{\frac{1}{\pi} \sqrt{1 - \left(\frac{x}{2}\right)^2}}{if $|x| \le 2$}{0}{otherwise;}\ee to obtain the standard semi-circle law we need to normalize the eigenvalues by $2\sqrt{N}$ and not $\sqrt{N}$.

%One way to prove results such as this is through Markov's method of moments, where the convergence of measures follows from the convergence of the $k$\textsuperscript{th} moment averaged over the ensemble to the $k$\textsuperscript{th} moment of the semi-circle, with some control on the rate of convergence. The central difficulty in the subject is that, while we want to understand the behavior of the eigenvalues, we have information only about the matrix elements; fortunately, the two quantities are connected through the eigenvalue trace lemma, which allows us to express powers of the eigenvalues in terms of the trace of powers of our matrix. This leads to \bea k{\rm \textsuperscript{th}\ moment\ of\ } \mu_{A,N}(x)\  & \ = \ & \frac{\sum_{i=1}^N\lambda_i(A)^k}{N^{\frac{k}2+1}} \ = \ \frac{{\rm Tr}(A^k)}{N^{\frac{k}2+1}}, \eea and reduces the problem to analyzing the average of the trace of powers of the matrices.

As the eigenvalues of any real symmetric matrix are real, we can ask whether or not a limiting distribution exists for the density of normalized eigenvalues for other ensembles. There are many interesting families to study. McKay \cite{McK} proved that the limiting spectral measure for adjacency matrices attached to $d$-regular graphs on $N$ vertices exists, and as $N \to \infty$, for almost all such graphs the associated measures converge to Kesten's measure  \be \twocase{f_{{\rm Kesten},d}(x) \ = \
}{\frac{d}{2\pi(d^2-x^2)} \sqrt{4(d-1) - x^2},}{$|x| \le
2\sqrt{d-1}$}{0}{otherwise} \ee (note that the measures may be scaled such that as $d\to\infty$ they converge to the semi-circle distribution).

This example and its behavior are typical for what we hope to find and prove. Specifically, we are looking for a thin subfamily that has different behavior but, as we fatten the ensemble to the full family of all real symmetric matrices, the limiting spectral measure converges to the semi-circle. Numerous researchers have studied a multitude of special, patterned matrices; we do not attempt to do this vast subject justice, but rather concentrate on a few ensembles closely related to our work.

All of the ensembles we consider here are linked ensembles (see \cite{BanBo}).  A linked ensemble of $N \times N$ matrices is specified by a link function $L_N : \{1,2,\dots,N\}^2 \to S$ to some set $S$.  To $s \in S$, assign random variables $x_s$ which are independent, identically distributed from a fixed probability distribution with density $p$ having mean $0$, variance $1$, and finite higher moments.  Set the $(i,j)^{\rm th}$ entry of the matrix $a_{i,j} := x_{L_N(i,j)}$.\footnote{For general linked ensembles, it may make more sense to weight the random variables by how often they occur in the matrix: $a_{i,j} := c_N |L_N^{-1}(\{L_N(i,j)\})|^{-1} x_{L_N(i,j)}$.  For the real symmetric ensemble, this corresponds to weighting the entries along the diagonal by $2$.  In that case, and for the ensembles we examine here, this modification changes only lower order terms in the calculations of the limiting spectral measure.}
For some linked ensembles, including those we examine here, it is be more convenient to specify the ensemble not by the link function, but by the equivalence relation $\sim$ it induces on $\{1,2,\dots,N\}^2$.  A link function may be uncovered as the quotient map to the set of equivalence classes $\{1,2,\dots,N\}^2 \surj \{1,2,\dots,N\}^2/\sim$. For example, the real symmetric ensemble is specified by the equivalence relation $(i,j) \sim (j,i)$.

%A convenient link function is $L(i,j) = (\min{(i,j)}, \max{(i,j)})$.

One interesting thin linked ensemble is that of real symmetric Toeplitz matrices, which are constant along its diagonals. The limiting measure is close to but not a Gaussian (see \cite{BCG,BDJ,HM}); however, in \cite{MMS} the sub-ensemble where the first row is replaced with a palindrome is shown to have the Gaussian as its limiting measure. While the approach in \cite{MMS} involves an analysis of an associated system of Diophantine equations, using Cauchy's interlacing property one can show that this problem is equivalent to determining the limiting spectral measure of symmetric circulant matrices (also studied in \cite{BM}).

While these and other ensembles related to circulant, Toeplitz, and patterned matrices are a very active area \cite{BasBo1,BasBo2,BanBo,BCG,BH,BM,BDJ,HM,MMS}, of particular interest to us are ensembles of patterned matrices with a variable parameter controlling the symmetry. We desire to deform a family of matrices, starting off with a highly structured family and ending with the essentially structureless case of real symmetric matrices. This is in contrast to some other work, such as Kargin \cite{Kar} (who studied banded Toeplitz matrices) and Jackson, Miller, and Pham \cite{JMP} (who studied Toeplitz matrices whose first row had a fixed but arbitrarily number of palindromes). In these cases the ensembles are converging to the full Toeplitz ensemble (either as the band grows or the number of palindromes decreases).

Our main ensemble is what we call the ensemble of \textbf{$m$-block circulant matrices}. A real symmetric circulant matrix (also called a symmetric circulant matrix) is a real symmetric matrix that is constant along diagonals and has first row $(x_0, x_1, x_2, \dots, x_2, x_1)$. Note that except for the main diagonal, a diagonal of length $N-k$ in the upper right is paired with a diagonal of length $k$ in the bottom left, and all entries along these two diagonals are equal. We study block Toeplitz and circulant matrices with $m \times m$ blocks.  The diagonals of such matrices are periodic of period $m$.

%We introduce a period parameter by requiring each of these paired diagonals to be periodic with $m$ independent random variables in the same order $N/m$ times, except for the small number of exceptions forced by being real symmetric (we'll discuss this fully below). We take a similar construction for the symmetric $m$-block Toeplitz matrices, although in this case each random variable need not appear the same number of times.

\begin{defi}[$m$-Block Toeplitz and Circulant Matrices]
Let $m|N$. An $N \times N$ real symmetric $m$-block Toeplitz matrix is a Toeplitz matrix of the form %CHECKED--IT'S CORRECT.  WE WANT TO DEFINE THEM IN THE NON-SYMMETRIC CASE BECAUSE IT'S POSSIBLE THAT THE REPN THEORY APPROACH COULD LEAD TO RESULTS ABOUT THAT CASE.
\be
\left(\begin{array}{ccccc}
B_0 & B_1 & B_2 & \cdots & B_{\nicefrac{N}{m}-1}\\
B_{-1} & B_0 & B_1 & \cdots & B_{\nicefrac{N}{m}-2}\\
B_{-2} & B_{-1} & B_0 & \cdots & B_{\nicefrac{N}{m}-3}\\
\vdots & \vdots & \vdots & \ddots & \vdots\\
B_{1-\nicefrac{N}{m}} & B_{2-\nicefrac{N}{m}} & B_{3-\nicefrac{N}{m}} & \cdots & B_0
\end{array}\right),
\nonumber\ee
with each $B_i$ an $m\times m$ real matrix.  An $m$-block circulant matrix is one of the above form for which $B_{-i} = B_{n-i}$.
\end{defi}

We investigate real symmetric $m$-block Toeplitz and circulant matrices.  In such matrices, a generic set of paired diagonals is composed of $m$ independent entries, placed periodically; however, as the matrix is real symmetric, this condition occasionally forces additional entries on the paired diagonals of length $\nicefrac{N}{2}$ to be equal.

For example, an $8 \times 8$ symmetric $2$-block Toeplitz matrix has the form
\be
\left(\begin{array}{cc|cc|cc|cc}
c_0 & c_1 & c_2 & c_3 & c_4 & c_5 & c_6 & c_7\\
c_1 & d_0 & d_1 & d_2 & d_3 & d_4 & d_5 & d_6\\ \cline{1-8}
c_2 & d_1 & c_0 & c_1 & c_2 & c_3 & c_4 & c_5\\
c_3 & d_2 & c_1 & d_0 & d_1 & d_2 & d_3 & d_4\\ \cline{1-8}
c_4 & d_3 & c_2 & d_1 & c_0 & c_1 & c_2 & c_3\\
c_5 & d_4 & c_3 & d_2 & c_1 & d_0 & d_1 & d_2\\ \cline{1-8}
c_6 & d_5 & c_4 & d_3 & c_2 & d_1 & c_0 & c_1\\
c_7 & d_6 & c_5 & d_4 & c_3 & d_2 & c_1 & d_0\\
\end{array}\right),
\ee
while a $6 \times 6$ and an $8 \times 8$ symmetric $2$-block circulant matrix have the form
\be
\left(\begin{array}{cc|cc|cc}
c_0 & c_1 & c_2 & c_3 & c_2 & d_1\\
c_1 & d_0 & d_1 & d_2 & c_3 & d_2\\ \cline{1-6}
c_2 & d_1 & c_0 & c_1 & c_2 & c_3\\
c_3 & d_2 & c_1 & d_0 & d_1 & d_2\\ \cline{1-6}
c_2 & c_3 & c_2 & d_1 & c_0 & c_1\\
d_1 & d_2 & c_3 & d_2 & c_1 & d_0
\end{array}\right),\ \ \ \ \ \ \
\left(\begin{array}{cc|cc|cc|cc}
c_0 & c_1 & c_2 & c_3 & c_4 & d_3 & c_2 & d_1\\
c_1 & d_0 & d_1 & d_2 & d_3 & d_4 & c_3 & d_2\\ \cline{1-8}
c_2 & d_1 & c_0 & c_1 & c_2 & c_3 & c_4 & d_3\\
c_3 & d_2 & c_1 & d_0 & d_1 & d_2 & d_3 & d_4\\ \cline{1-8}
c_4 & d_3 & c_2 & d_1 & c_0 & c_1 & c_2 & c_3\\
d_3 & d_4 & c_3 & d_2 & c_1 & d_0 & d_1 & d_2\\ \cline{1-8}
c_2 & c_3 & c_4 & d_3 & c_2 & d_1 & c_0 & c_1\\
d_1 & d_2 & d_3 & d_4 & c_3 & d_2 & c_1 & d_0\\
\end{array}\right);
\ee Note for the $6\times 6$ matrix that being real symmetric forces the paired diagonals of length $\nicefrac{N}{2}$ (i.e., 3) to have just one and not two independent random variables. An equivalent viewpoint is that each `wrapped' diagonal is periodic with period $m$ and has $m$ distinct random variables. Note that the diagonals are wrapped toroidally, and each such diagonal has $N$ elements.

Clearly if $m=1$ these ensembles reduce to the previous cases, and as $m\to\infty$ they approach the full family of real symmetric matrices; in other words, the circulant or Toeplitz structure vanishes as $m\to\infty$, but for any finite $m$ there is additional structure. The goal of this paper is to determine the limiting spectral measures for these families and to quantify how the convergence to the semi-circle depends on $m$.  We find an explicit closed form expression for the limiting spectral density of the $m$-block circulant family as a product of a Gaussian and a degree $2m-2$ polynomial.

\subsection{Results}

Before stating our results, we must define the probability spaces where our ensemble lives and state the various types of convergence that we can prove. We provide full details for the $m$-block circulant matrices, as the related Toeplitz ensemble is similar. The following definitions and set-up are standard, but are included for completeness. We paraphrase from \cite{MMS,JMP} with permission.

Fix $m$ and for each integer $N$ let $\Omega_{m,N}$ denote the set of $m$-block circulant matrices of dimension $N$.  Define an equivalence relation $\simeq$ on $\{1,2,\dots,N\}^2$.  Say that $(i,j) \simeq (i^\p,j^\p)$ if and only if $a_{i j} = a_{i^\p j^\p}$ for all $m$-block circulant matrices, in other words, if
\begin{itemize}
\item
$j-i \con j^\p-i^\p \Mod{N} \mbox{ and } i \con i^\p \Mod{m}, \mbox{ or}$
\item
$j-i \con -(j^\p-i^\p) \Mod{N} \mbox{ and } i \con j^\p \Mod{m}.$
\end{itemize}

Consider the quotient $\{1,2,\dots,N\}^2 \surj \{1,2,\dots,N\}^2/\simeq$.  This induces an injection $\R^{\{1,2,\dots,N\}^2/\simeq}$ $\inj$ $\R^{N^2}$.  The set $\R^{\{1,2,\dots,N\}^2/\simeq}$ has the structure of a probability space with the product measure of $p(x)\,dx$ with itself $|\{1,2,\dots,N\}^2/\simeq|$ times, where $dx$ is Lebesgue measure.  We define the probability space $(\Omega_{m,N},\f_{m,N},\pp_{m,N})$ to be its image in $\R^{N^2} = M_{N^2}(\R)$ under the injection, with the same distribution.

%Let $e_i = 1$ if $N/m$ is even and  $i \geq N/4$, and let $e_i = 0$ otherwise. \textbf{(CHECK THIS CONDITION!)} We define the probability space $(\Omega_{m,N},\f_{m,N},\pp_{m,N})$ by setting \bea\label{eq:probspace} & &
%\pp_{m,N}\left(\left\{A_N\in\Omega_{m,N}: b_{i,j}(A_N) \in [\ga_{i,j}, \gb_{i,j}]\
%{\rm for}\ i \in \{1,\dots,m\}, j \in \{0,\dots, \lfloor N/2 \rfloor - e_i \}  \right\}\right)
%\nonumber\\ & & \ \ \ \ \ \ \ \ := \ \prod_{i=1}^m \prod_{j=0}^{\lfloor N/2 \rfloor - e_i}
%\int_{x_{i,j}=\ga_{i,j}}^{\gb_{i,j}} p(x_{i,j})\,dx_{i,j}, \eea where each $dx_{i,j}$ is
%Lebesgue measure.
To each $A_N\in \Omega_{m,N}$ we attach a
measure by placing a point mass of size $\nicefrac{1}{N}$ at each normalized
eigenvalue $\lambda_i(A_N)$: \be\label{eq:normspacingmeasure}
\mu_{m, A_N}(x)dx \ = \ \frac{1}{N} \sum_{i=1}^N \delta\left( x -
\frac{\lambda_i(A_N)}{\sqrt{N}} \right)dx, \ee where $\delta(x)$ is the
standard Dirac delta function; see Footnote \ref{foot:sizeevaluesfull} for an explanation of the normalization factor equaling $\sqrt{N}$. We call $\mu_{m, A_N}$ the normalized
spectral measure associated with $A_N$.

%\footnote{From the eigenvalue trace lemma ($\text{Trace}(A_N^2) = \sum_i \lambda_i^2(A_N)$) and the Central Limit Theorem, we see that the eigenvalues of $A_N$ are of order $\sqrt{N}$, independent of $m$. This is because $\text{Trace}(A_N^2) = \sum_{i,j=1}^N a_{ij}^2$, and since each $a_{ij}$ is drawn from a mean $0$, variance $1$ distribution, $\text{Trace}(A_N^2)$ is of size $N^2$. This suggests the appropriate scale for normalizing the eigenvalues is to divide each by $\sqrt{N}$.}

\begin{defi}[Normalized empirical spectral
distribution]\label{defi:nesd} Let $A_N \in \Omega_{m,N}$ have eigenvalues $\lambda_N \ge \cdots \ge
\lambda_1$. The normalized empirical spectral distribution (the
empirical distribution of normalized eigenvalues) $F_m^{\nicefrac{A_N}{\sqrt{N}}}$
is defined by \be F_m^{\nicefrac{A_N}{\sqrt{N}}}(x) \ = \frac{\#\{i \le N:
\nicefrac{\lambda_i}{\sqrt{N}} \le x\}}{N}. \ee \end{defi}

As $F_m^{\nicefrac{A_N}{\sqrt{N}}}(x) = \int_{-\infty}^x \mu_{m, A_N}(t)dt$, we see
that $F_m^{\nicefrac{A_N}{\sqrt{N}}}$ is the cumulative distribution function
associated to the measure $\mu_{n, A_N}$. We are interested in the behavior of a typical $F_m^{\nicefrac{A_N}{\sqrt{N}}}$ as we vary $A_N$ in our ensembles $\Omega_{m,N}$ as $N\to\infty$.

Consider any probability space $\Omega_m$ which has the $\Omega_{m,N}$ as quotients.  (The most obvious example is the independent product.)  This paper build on a line of papers \cite{HM, MMS, JMP} concerning various Toeplitz ensembles which fix $\Omega_m$ to be the space of $\N$-indexed strings of real numbers picked independently from $p$, with quotient maps to each $\Omega_{m,N}$ mapping a string to a matrix whose free parameters come from an initial segment of the right length.  There is no need for the specificities of this construction, so we consider the general case.

%\footnote{THING ABOUT Rokhlin?}
%
%As there is a one-to-one correspondence between $N\times N$ period $m$-block circulant matrices and $\R^{mN}$, we may study the more
%convenient infinite sequences. Thus our outcome space is \be \Omega_{m,\N}
%\ =\ \{b_{1,0},\dots, b_{m,0}, b_{1,1}, \dots, b_{m,N-1}\}, \ee and if $\omega = (\omega_0,\omega_1,\dots) \in
%\Omega_{m,\N}$ then \be {\rm Prob}(\omega_i \in [\alpha_i,\beta_i])\ =\
%\int_{\alpha_i}^{\beta_i} p(x_i)dx_i.\ee We denote elements of
%$\Omega_{m,\N}$ by $A$ to emphasize the correspondence with matrices,
%and we set $A_N$ to be the period $m$-block circulant
%matrix obtained by truncating $A$ by taking the first $Nm$ entries. We denote the probability space by
%$(\Omega_{m,\N},\f_{m, \N},\pp_{m, \N})$.
%

\begin{defi}[Limiting spectral distribution]\label{defi:lsd}
If as $N\to\infty$ we have $F_m^{\nicefrac{A_N}{\sqrt{N}}}$ converges in some
sense (for example, in probability or almost surely) to a distribution $F_m$,
then we say $F_m$ is the limiting spectral distribution of the
ensemble.
\end{defi}

We investigate the symmetric $m$-block Toeplitz and circulant ensembles. We may view these as structurally weakened real symmetric Toeplitz and circulant ensembles. When $m$ is $1$ we regain the Toeplitz (circulant) structure, while if $m=N$ we have the general real symmetric ensemble.  If $m$ is growing with the size of the matrix, we expect the eigenvalues to be distributed according to the semi-circle law, while for fixed $m$ we expect to see new limiting spectral distributions.

Following the notation of the previous subsection, for each integer $N$ we let $\Omega^{(T)}_{m,N}$ and $\Omega^{(C)}_{m,N}$ denote the probability space of real symmetric $m$-block Toeplitz and circulant matrices of dimension $N$, respectively. We now state our main results.

\begin{thm}[Limiting spectral measures of symmetric block Toeplitz and circulant ensembles]\label{thm:mainexpformulasconvergence} Let $m|N$.

\begin{enumerate}

%-\raisebox{2pt}{$t^2$}\huge/ \diagup \raisebox{-2pt}{$2m$}
\item The characteristic function of the limiting spectral measure of the symmetric $m$-block circulant ensemble is
\bea\label{phiform}
\phi_m(t) \ = \  \frac{1}{m} e^{-\nicefrac{t^2}{2m}}e^{-\nicefrac{t^2}{2m}} L_{m-1}^{(1)}\left(\frac{t^2}{m}\right) \ = \ e^{-\nicefrac{t^2}{2m}} M\left(m+1,2,-\nicefrac{t^2}{m}\right),
\eea
where $L_{m-1}^{(1)}$ is a generalized Laguerre polynomial and $M$ a confluent hypergeometric function.  The expression equals the spectral characteristic function for the $m \times m$ GUE.  The limiting spectral density function (the Fourier transform of $\phi_m$) is
\be
f_m(x)\ =\ \frac{ e^{-\nicefrac{mx^2}{2}}}{\sqrt{2\pi m}} \sum_{r=0}^{m-1} \frac{1}{(2r)!}
\left( \sum_{s=0}^{m-r} {m \choose r+s+1} \frac{(2r+2s)!}{(r+s)! s!} \left(-\foh \right)^s \right)
(m x^2)^r. \ee For any fixed $m$, the limiting spectral density is the product of a Gaussian and an even polynomial of degree $2m-2$, and has unbounded support.

\item If $m$ tends to infinity with $N$ (at any rate) then the limiting spectral distribution of the symmetric $m$-block circulant and Toeplitz ensembles, normalized by rescaling $x$ to $x/2$, converge to the semi-circle distribution; without the renormalization, the convergence is to a semi-ellipse, with density $f_{\rm Wig}$ (see \eqref{eq:densityfwigforwigner}).

\item As $m \to \infty$, the limiting spectral measures $f_m$ of the $m$-block circulant ensemble converge uniformly and in $L^p$ for any $p \ge 1$ to $f_{\rm Wig}$, with $|f_m(x) - f_{\rm Wig}(x)| \ll m^{-\nicefrac{2}{9}+\epsilon}$ for any $\epsilon>0$.

\item The empirical spectral measures of the $m$-block circulant and Toeplitz ensembles converge weakly and in probability to their corresponding limiting spectral measures, and we have almost sure convergence if $p$ is an even function.

\end{enumerate}

\end{thm}

Figure \ref{figure:plotsconvsemicircle} illustrates the convergence of the limiting measures to the semi-circle; numerical simulations (see Figures \ref{figure:hist2hist3}, \ref{figure:hist4hist8} and \ref{figure:hist1hist20}) illustrate the rapidity of the convergence. We see that even for small $m$, in which case there are only $\nicefrac{mN}{2}$ non-zero entries in the adjacency matrices (though these can be any of the $N^2-N$ non-diagonal entries of the matrix), the limiting spectral measure is close to the semi-circle. This behavior is similar to what happens with $d$-regular graphs, though in our case the convergence is faster and the support is unbounded for any finite $m$.
\begin{center}
\begin{figure}
\centering\includegraphics[scale=1.25]{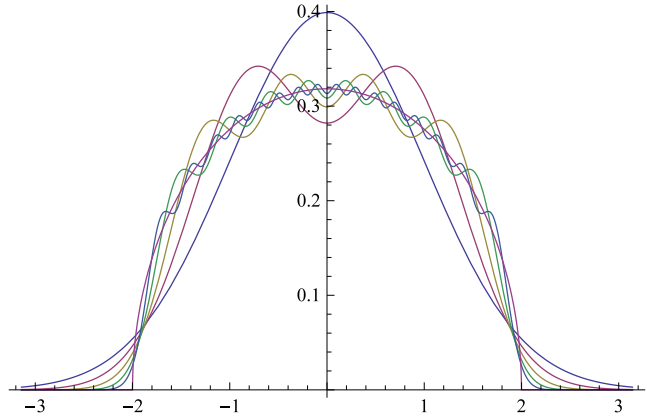}
\caption{Plots for $f_1, f_2, f_4, f_8, f_{16}$ and the semi-circle density.}\label{figure:plotsconvsemicircle}
\end{figure}
\end{center}
%\begin{center}
%\begin{figure}
%\centering\includegraphics[scale=.95]{KestenDandSemiCircle.eps}
%\caption{Plots of the normalized limiting spectral measure for $d$-regular graphs with $d=3, 4, 8$ and 16, and the semi-circle.}\label{figure:KestenDandSemiCircle}
%\end{figure}
%\end{center}

%Another interesting feature about our limiting measures are the oscillations in the densities due to the polynomial factors. Numerical simulations (see Figures \ref{figure:hist2hist3} and \ref{figure:hist4hist8}) provide striking support. For comparison purposes, in Figure \ref{figure:hist1hist20} we plot numerics for the $m=1$ case (which are circulant matrices, whose density is known by \cite{BM,MMS} to be Gaussian) against the $m=20$ case, where already we can see the semi-circular behavior.

\begin{figure}
%\centering\includegraphics[scale=1]{hist2.eps}\ \centering\includegraphics[scale=1]{hist3.eps}
\centering\includegraphics[scale=.7]{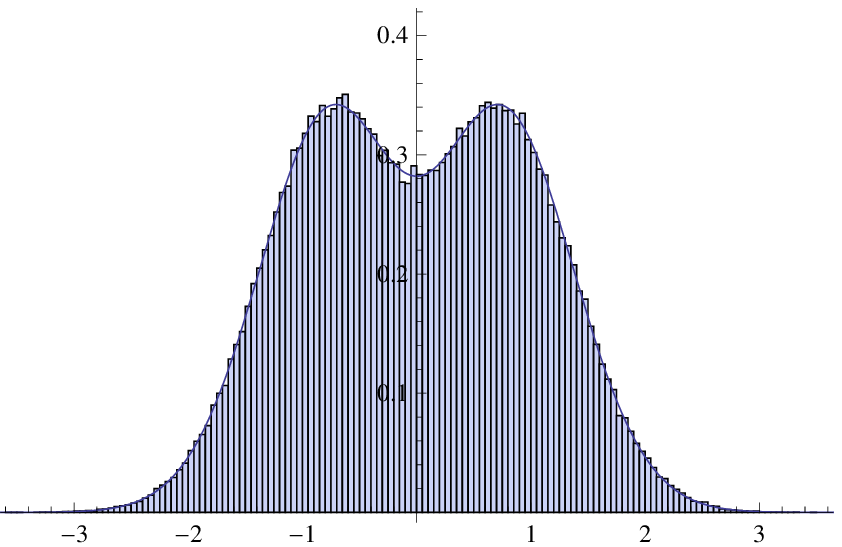}\ \centering\includegraphics[scale=.7]{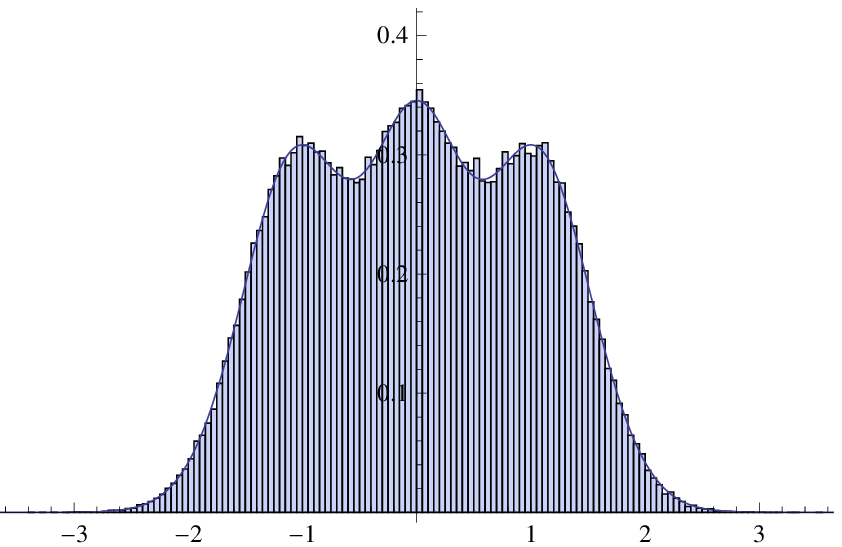}
\caption{(Left) Plot for $f_2$ and histogram of eigenvalues of 1000 symmetric period $2$-block circulant matrices of size $400 \times 400$. (Right) Plot for $f_3$ and histogram of eigenvalues of 1000 symmetric period $3$-block circulant matrices of size $402 \times 402$.}\label{figure:hist2hist3}
%\centering\includegraphics[scale=1]{hist4.eps}\ \centering\includegraphics[scale=1]{hist8.eps}
\centering\includegraphics[scale=.7]{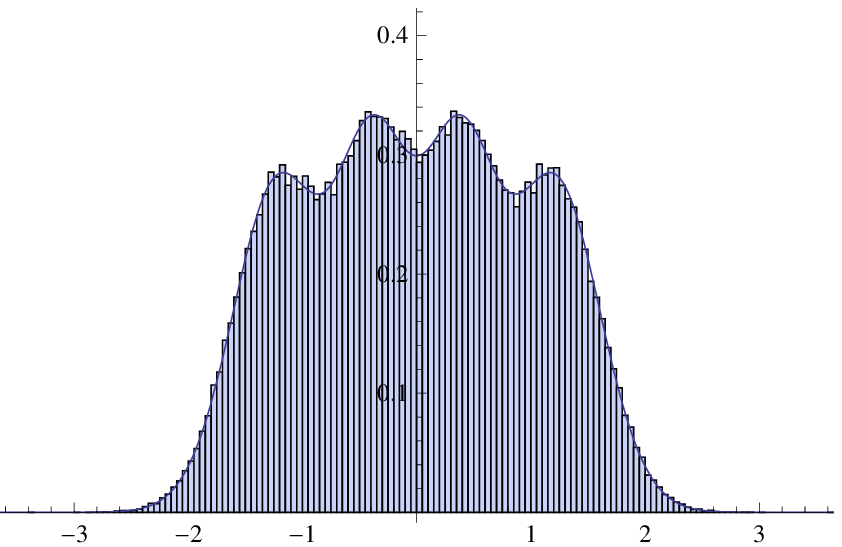}\ \centering\includegraphics[scale=.7]{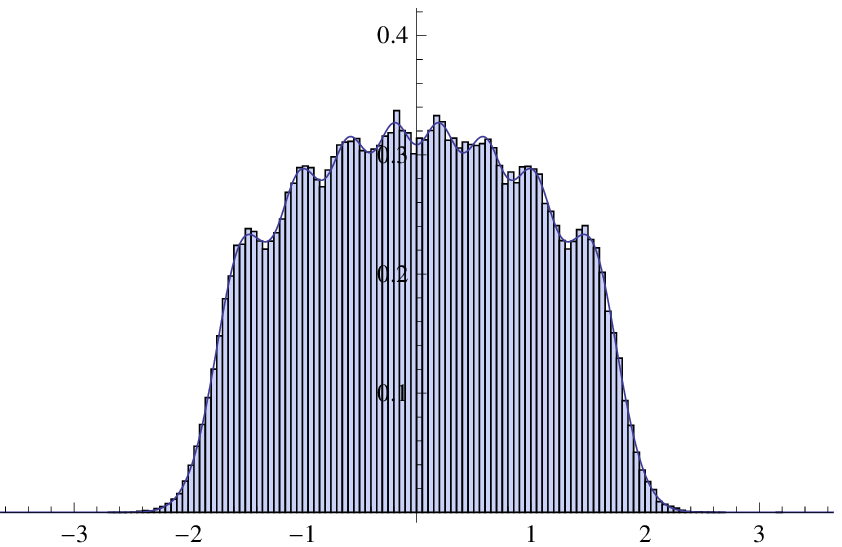}
\caption{(Left) Plot for $f_4$ and histogram of eigenvalues of 1000 symmetric period $4$-block circulant matrices of size $400 \times 400$. (Right) Plot for $f_8$ and histogram of eigenvalues of 1000 symmetric period $8$-block circulant matrices of size $400 \times 400$.}\label{figure:hist4hist8}
\centering\includegraphics[scale=.7]{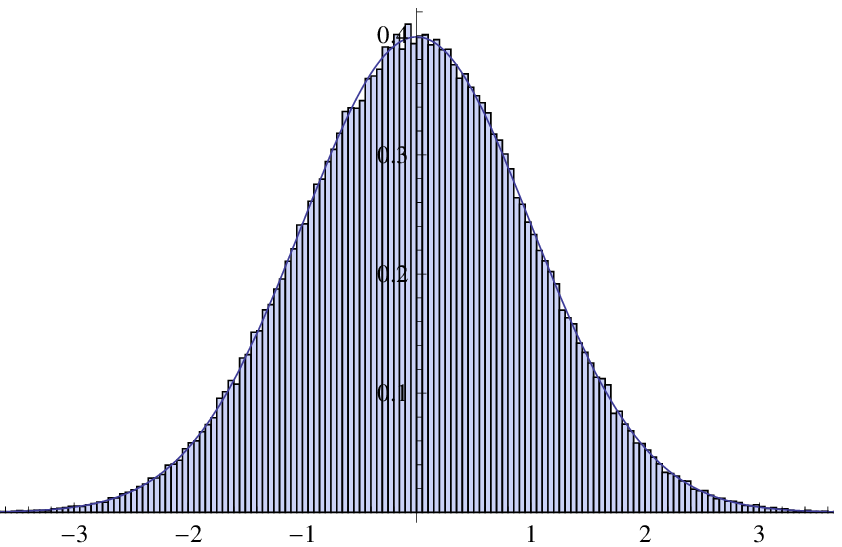}\ \centering\includegraphics[scale=.7]{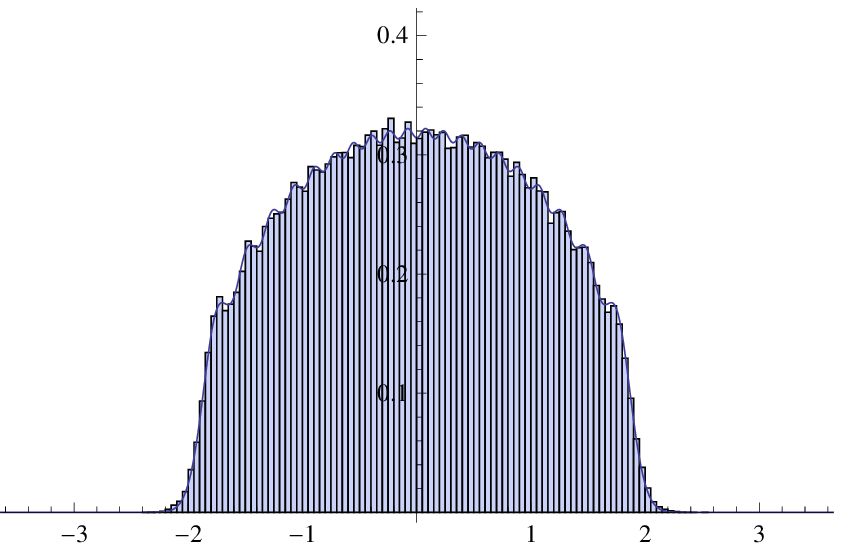}
\caption{(Left) Plot for $f_1$ and histogram of eigenvalues of 1000 symmetric period $1$-block circulant matrices of size $400 \times 400$. (Right) Plot for $f_{20}$ and histogram of eigenvalues of 1000 symmetric period $20$-block circulant matrices of size $400 \times 400$.}\label{figure:hist1hist20}
\end{figure}

%\begin{figure}
%\centering\includegraphics[scale=1]{hist20.eps}
%\caption{Plot for $f_{20}$ and histogram of eigenvalues of 1000 $20$-block circulant matrices of size $400 \times 400$.}
%\end{figure}

Finally, the limiting eigenvalue density for $m$-block circulant matrices is the same as the eigenvalue density of a certain Gaussian Hermitian ensemble.  Specifically, we consider $m \times m$ Hermitian matrices with off-diagonal entries picked independently from a complex Gaussian with density function $p(z) = \frac{1}{\pi}e^{-|z|^2}$, and diagonal entries picked independently from a real Gaussian of mean $0$ and variance $1$.  We provide a heuristic for why these densities are the same in \S\ref{sec:repn}; see also \cite{Zv} (especially Section 5.2) for a proof.\\

Our results generalize to related ensembles. For example, the (wrapped) diagonals of our $m$-block circulant ensembles have the following structure (remember we assume $m|N$): \be (b_{1,j}, b_{2,j}, \dots, b_{m,j}, b_{1,j}, b_{2,j}, \dots, b_{m,j}, \dots, b_{1,j}, b_{2,j}, \dots, b_{m,j}). \ee Note that we have a periodic repeating block of size $m$ with $m$ independent random variables; for brevity, we denote this structure by \be (d_1, d_2, \dots, d_m). \ee Similar arguments handle other related ensembles, such as the subfamily of period $m$--ciculant matrices in which some entries within the period are forced to be equal. Interesting comparisons are $(d_1, d_2) = (d_1,d_2,d_1,d_2)$ versus $(d_1, d_1, d_2, d_2)$ or $(d_1, d_2, d_2, d_1)$.  While it is a natural guess that the limiting spectral measure is determined solely by the frequency at which each letter appears, this is false.

\begin{thm}\label{thm:extlsd} Let $\mathcal{P} = (d_{i_1}, d_{i_2}, \dots, d_{i_m})$ where each $d_{i_j} \in \{d_1,\dots,d_\nu\}$ and each $d_i$ occurs exactly $r_i$ times in the pattern $\mathcal{P}$, with $r_1 +\cdots+r_\nu = m$; equivalently, $\mathcal{P}$ is a permutation of $(d_1,\dots,d_1,d_2,\dots,d_2, \dots, d_\nu,\dots,d_\nu)$ with $r_i$ copies of $d_i$. Modify the $N\times N$ period $m$-block circulant matrices by replacing the pattern $(d_1,d_2,\dots,d_m)$ with $\mathcal{P}$ (remember $m|N$). Then for any $\mathcal{P}$ as $N\to\infty$ the limiting spectral measure exists. The resulting measure does not depend solely on the frequencies of the letters in the pattern but also on their locations; in particular, while the fourth moments of the measures associated to $\{d_1,d_2,d_1,d_2\}$ and $\{d_1,d_1,d_2,d_2\}$ are equal (interestingly, the fourth moment of any pattern only depends on the frequencies), the sixth moments differ.
\end{thm}

We prove our main results using the method of moments. As the proof of Theorem \ref{thm:extlsd} is similar to that of Theorem \ref{thm:mainexpformulasconvergence}, we just sketch the ideas and computations in Appendix \ref{sec:appwentao}.  For our ensembles, we first show that the average of the $k$\textsuperscript{th} moments over our ensemble converge to the moments of a probability density. By studying the variance or fourth moment of the difference of the moments of the empirical spectral measures and the limits of the average moments, we obtain the various types of convergence by applications of Chebyshev's inequality and the Borel-Cantelli Lemma. These arguments are similar to previous works in the literature, and yield only the existence of the limiting spectral measure.

Unlike other works for related ensembles, however, we are able to obtain explicit closed form expressions for the moments for the symmetric $m$-block circulant ensemble. This should be compared to the Toeplitz ensemble case, where previous studies could only relate these moments to volumes of Eulerian solids or solutions to systems of Diophantine equations. Similar to other ensembles, we show that the only contribution in the limit is when $k = 2\ell$ and the indices are matched in pairs with opposite orientation. We may view this as a $2\ell$-gon with vertices $(i_1,i_2)$, $(i_2,i_3)$, $\dots$, $(i_{2\ell},i_1)$. The first step is to note that when $m=1$, similar to the circulant and palindromic Toeplitz ensembles, each matching contributes 1; as there are $(2\ell-1)!!$ ways to match $2\ell$ objects in pairs, and as $(2\ell-1)!!$ is the $2\ell$\textsuperscript{th} moment of the standard normal, this yields the Gaussian behavior. For general $m$, the key idea is to look at the dual picture. Instead of matching indices we match edges. In the limit as $N\to\infty$, the only contribution occurs when the edges are matched in pairs with opposite orientation. Topologically, these are exactly the pairings which give orientable surfaces. If $g$ is the genus of the associated surface, then the matching contributes $m^{-2g}$. Harer and Zagier \cite{HarZa} determined formulas for $\varepsilon_g(\ell)$, the number of matchings that form these orientable surfaces. This yields the $N\to\infty$ limit of the average $2\ell$\textsuperscript{th} moment is \be \sum_{g=0}^{\lfloor \nicefrac{\ell}{2} \rfloor}\varepsilon_g(\ell) m^{-2g}.\ee After some algebra, we express the characteristic function (which is the inverse Fourier transform; see Footnote \ref{foot:charfnfootnote}) of the limiting spectral measure as a certain term in the convolution of the associated generating function of the $\varepsilon_g$'s and the normal distribution, which we can compute using Cauchy's residue theorem. Taking the Fourier transform (appropriately normalized) yields an explicit, closed form expression for the density. We note that the same formulas arise in investigations of the moments for Gaussian ensembles; see Section 1.6 of \cite{Fo} and \cite{Zv} (as well as the references therein) for additional comments and examples.

The paper is organized as follows. In \S\ref{sec:trace} we describe the method of proof and derive useful expansions for the moments in terms of quantities from algebraic topology. We use these in \S\ref{sec:determiningevenmoments} to determine the limiting spectral measures, and show convergence in \S\ref{sec:convergence}. We conclude in \S\ref{sec:futurework} with a description of future work and related results. Appendix \ref{sec:convrate} provides some needed estimates for proving the rate of convergence in Theorem \ref{thm:mainexpformulasconvergence}, and we conclude in Appendix \ref{sec:appwentao} with a discussion of the proof of Theorem \ref{thm:extlsd}  (see \cite{Xi} for complete details).

%%%%%%%%%%%%%%%%%%%%%%%%%%%%%%%%%%%%%%%%%%%%%%%%%%%%%%%%%%%%%%%%%%%%%%%%%%%%%%%%%%%%%%%%%%%%%%%%
%%%%%%%%%%%%%%%%%%%%%%%%%%%%%%%%%%%%%%%%%%%%%%%%%%%%%%%%%%%%%%%%%%%%%%%%%%%%%%%%%%%%%%%%%%%%%%%%

\section{Moments Preliminaries}\label{sec:trace}

In this section we investigate the moments of the associated spectral measures. We first describe the general framework of the convergence proofs and then derive useful expansions for the average moments for our ensemble for each $N$ (Lemma \ref{lem:oddmomentsandevenformula}). The average odd moments are easily seen to vanish, and we find a useful expansion for the $2k$\textsuperscript{th} moment in Lemma \ref{lem:formulaforevenmomentsepsilon}, relating this moment to the number of pairings of the edges of a $2k$-gon giving rise to a genus $g$ surface

\subsection{Markov's Method of Moments}

%We use Markov's Method of Moments to find the limiting spectral density for our ensemble. In particular, we look at traces of powers of a typical matrix in this ensemble.

For the eigenvalue density of a particular $N \times N$ symmetric $m$-block circulant matrix $A$, we use the redundant notation $\mu_{m, A, N}(x)\,dx$ (to emphasize the $N$ dependence), setting
\be
\mu_{A,N}(x)\,dx\ :=\ \frac{1}{N} \sum_{i=1}^N \delta\left(x-\frac{\lambda_i(A)}{\sqrt{N}}\right)\,dx.
\ee
To prove Theorem \ref{thm:mainexpformulasconvergence}, we must show

\begin{enumerate}

\item as $N\to\infty$ a typical matrix has its spectral measure close to the system average;

\item these system averages converge to the claimed measures.

\end{enumerate}

The second claim follows easily from Markov's Method of Moments, which we now briefly describe. To each integer $k \ge 0$ we define the random variable $X_{k;m,N}$ on
$\Omega_{m}$ by \be X_{k;m,N}(A) \ = \ \int_{-\infty}^\infty x^k
\,dF_m^{\nicefrac{A_N}{\sqrt{N}}}(x); \ee note this is the $k$\textsuperscript{{\rm th}}
moment of the measure $\mu_{m, A, N}$.

Our main tool to understand the average over all $A$ in our ensemble of the $F_m^{\nicefrac{A_N}{\sqrt{N}}}$'s is the Moment
Convergence Theorem (see \cite{Ta} for example); while the analysis in \cite{MMS} was simplified by the fact that the convergence was to the standard normal, similar arguments (see also \cite{JMP}) hold in our case as the growth rate of the moments of our limiting distribution implies that the moments uniquely determine a probability distribution.

\begin{thm}[Moment Convergence Theorem]\label{thm:momct} Let $\{F_N(x)\}$ be a
sequence of distribution functions such that the moments \be M_{k;N}
\ = \ \int_{-\infty}^\infty x^k dF_N(x) \ee exist for all $k$. Let $\{M_k\}_{k=1}^\infty$ be a sequence of moments that uniquely determine a probability distribution, and denote the cumulative distribution function by $\Psi$. If
$\lim_{N\to\infty} M_{k,N} = M_k$ then $\lim_{N\to\infty} F_N(x) =
\Psi(x)$. \end{thm}

We will see that the average moments uniquely determine a measure, and will be left with proving that a typical matrix has a spectral measure close to the system average. The $n$\textsuperscript{th} moment of $A$'s measure, given by integrating $x^n$ against $\mu_{m,A,N}$, is
\be
M_{n;m}(A,N)\ =\ \frac{1}{N}\sum_{i=1}^N \left(\frac{\lambda_i(A)}{\sqrt{N}}\right)^n\ =\ \frac{1}{N^{\nicefrac{n}{2}+1}} \sum_{i=1}^N \lambda_i^n(A).
\ee

We define \be
M_{n;m}(N)\ :=\ \E(M_{n;m}(A,N)),
\ee
and set
\be
M_{n;m}\ :=\ \lim_{N \to \infty} M_{n;m}(N)
\ee
(we'll show later that the limit exists). By $\E(M_{n;m}(A,N))$, we mean the expected value of $M_{n;m}(A,N)$ for a random symmetric $m$-block circulant matrix $A \in \Omega_{m,N}$.

%, where we pick entries $a_{ij}$ i.i.d.r.v.. from $p$ for $1 \leq i \leq m$ and $i \leq j \leq i+\lfloor\frac{N}{2}\rfloor-e_i$ and fill in the rest of the matrix as the symmetries force us to.  We'll also assume that $m|N$ to simplify later counting arguments.

\subsection{Moment Expansion}

We use a standard method to compute the moments.  By the eigenvalue trace lemma,
\begin{equation}
\Tr(A^n)\ =\ \sum_{i=1}^N \lambda_i^n,
\end{equation}
so
\be
M_{n;m}(A,N)\ =\ \frac{1}{N^{\nicefrac{n}{2}+1}} \Tr(A^n).
\ee
Expanding out $\Tr(A^n)$,
\be
M_{n;m}(A,N)\ =\ \frac{1}{N^{\nicefrac{n}{2}+1}} \sum_{1 \leq i_1,\dots,i_n \leq N} a_{i_1 i_2} a_{i_2 i_3} \cdots a_{i_n i_1},
\ee
so by linearity of expectation,
\be\label{moment}
M_{n;m}(N)\ =\ \frac{1}{N^{\nicefrac{n}{2}+1}} \sum_{1 \leq i_1,\dots,i_n \leq N} \E(a_{i_1 i_2} a_{i_2 i_3} \cdots a_{i_n i_1}).
\ee

Recall that we've defined the equivalence relation $\simeq$ on $\{1$, $2$, $\dots$, $N\}^2$ by $(i,j) \simeq (i^\p,j^\p)$ if and only if $a_{i j} = a_{i^\p j^\p}$ for all real symmetric $m$-block circulant matrices.  That is, $(i,j) \simeq (i^\p,j^\p)$ if and only if
\begin{itemize}
\item
$j-i \con j^\p-i^\p \Mod{N} \mbox{ and } i \con i^\p \Mod{m}, \mbox{ or}$
\item
$j-i \con -(j^\p-i^\p) \Mod{N} \mbox{ and } i \con j^\p \Mod{m}.$
\end{itemize}

For each term in the sum in (\ref{moment}), $\simeq$ induces an equivalence relation $\sim$ on $\{(1,2)$, $(2,3)$, $\dots$, $(n,1)\}$ by its action on $\{(i_1,i_2),(i_2,i_3),\dots,(i_n,i_1)\}$.   Let $\eta(\sim)$ denote the number of $n$-tuples with $0 \leq i_1, \dots, i_n \leq N$ whose indices inherit $\sim$ from $\simeq$.
Say $\sim$ splits up $\{(1,2)$, $(2,3)$, $\dots$, $(n,1)\}$ into equivalence classes with sizes $d_1(\sim),\dots,d_l(\sim)$.
Because the entries of our random matrices are independent identically distributed,
\be
\E(a_{i_1 i_2} a_{i_2 i_3} \cdots a_{i_n i_1})\ =\ m_{d_1(\sim)} \cdots m_{d_l(\sim)},
\ee
where the $m_d$ are the moments of $p$.
Thus, we may write
\be\label{squigglesum}
M_{n;m}(N)\ =\ \frac{1}{N^{\nicefrac{n}{2}+1}} \sum_{\sim} \eta(\sim) m_{d_1(\sim)} \cdots m_{d_l(\sim)}.
\ee

As $p$ has mean $0$, $m_{d_1(\sim)} \cdots m_{d_l(\sim)} = 0$ unless all of the $d_j$ are greater than $1$.  So all the terms in the above sum vanish except for those coming from a relation $\sim$ which matches at least in pairs.

The $\eta(\sim)$ denotes the number of solutions modulo $N$ the following system of Diophantine equations: Whenever $(s,s+1) \sim (t,t+1)$,
\begin{itemize}
\item
$i_{s+1}-i_s \con i_{t+1}-i_t \Mod{N} \mbox{ and } i_s \con i_t \Mod{m}, \mbox{ or}$
\item
$i_{s+1}-i_s \con -(i_{t+1}-i_t) \Mod{N} \mbox{ and } i_s \con i_{t+1} \Mod{m}.$
\end{itemize}

This system has at most $2^{n-l} N^{l+1}$ solutions, a bound we obtain by completely ignoring the $\Mod{m}$ constraints (see also \cite{MMS}).  Specifically, we pick one difference $i_{s+1}-i_s$ from each congruence class of $\sim$ freely, and we are left with at most $2$ choices for the remaining ones.  Finally, we pick $i_1$ freely, and this now determines all the $\ds i_s = i_1 + \sum_{s^\p < s} (i_{s^\p+1}-i_{s^\p})$. This method will not always produce a legitimate solution, even without the $\Mod{m}$ constraints, but it suffices to give an upper bound on the number of solutions.

When $n$ is odd, say $n = 2k+1$, then $l$ is at most $k$.  Thus $\frac{1}{N^{\nicefrac{n}{2}+1}}\eta(\sim) \leq \frac{1}{N^{k+\nicefrac{3}{2}}}2^{n-l} N^{l+1} \leq \frac{1}{N^{k+\nicefrac{3}{2}}}2^{n-l} N^{k+1} = \frac{1}{\sqrt{N}}2^{n-l} = O_n\left(\frac{1}{\sqrt{N}}\right)$.
This implies the odd moments vanish in the limit, as
\be
M_{2k+1;m}(N)\ =\ O_k\left(\frac{1}{\sqrt{N}}\right).
\ee

When $n$ is even, say $n=2k$, then $l$ is at most $k$.  If $l < k$, then $l \leq k-1$, and we have, similar to the above, $\frac{1}{N^{\nicefrac{n}{2}+1}}\eta(\sim)$ $\leq$ $\frac{1}{N^{k+1}}2^{n-l} N^{l+1}$ $\leq$ $\frac{1}{N^{k+1}}2^{n-l} N^{k}$ $=$ $\frac{1}{N}2^{n-l}$ $=$ $O_n\left(\frac{1}{N}\right)$.
If $l=k$, then the entries are exactly matched in pairs, that is, all the $d_j = 2$. As $p$ has variance $1$ (i.e., $m_2 = 1$), the formula for the even moments, (\ref{squigglesum}), becomes
\be
M_{2k;m}(N)\ =\ \frac{1}{N^{k+1}} \sum_{\s} \eta(\s) + O_k\left(\frac{1}{N}\right).
\ee
We've changed notation slightly.  The sum is now over pairings $\s$ on $\{(1,2)$, $(2,3)$, $\dots$, $(n,1)\}$, which we may consider as functions (specifically, involutions with no fixed points) as well as equivalence relations. We have thus shown

\begin{lem}\label{lem:oddmomentsandevenformula}
For the ensemble of symmetric $m$-block circulant matrices,
\bea M_{2k+1;m}(N) & \ =\ & O_k\left(\frac{1}{\sqrt{N}}\right) \nonumber\\ M_{2k;m}(N) & \ =\ & \frac{1}{N^{k+1}} \sum_{\s} \eta(\s) + O_k\left(\frac{1}{N}\right), \eea where the sum is over pairings $\s$ on $\{(1,2)$, $(2,3)$, $\dots$, $(n,1)\}$. In particular,
as $N\to\infty$ the average odd moment is zero.
\end{lem}

\subsection{Even Moments}

We showed the odd moments go to zero like $\nicefrac{1}{\sqrt{N}}$ as $N \to \infty$; we now calculate the $2k$\textsuperscript{th} moments.
From Lemma \ref{lem:oddmomentsandevenformula}, the only terms which contribute in the limit are those in which the $a_{i_s i_{s+1}}$'s are matched in pairs.  We can think of the pairing as a pairing of the edges of a $2k$-gon with vertices $1,2,\dots,2k$ and edges $(1,2),(2,3),\dots,(2k,1)$.  The vertices are labeled $i_1,\dots,i_{2k}$ and the edges are labeled $a_{i_1 i_2},\dots,a_{i_{2k} i_1}$. See Figure \ref{fig:pairingnoorientation}.

\begin{figure}
\begin{center}
\includegraphics[scale=.34]{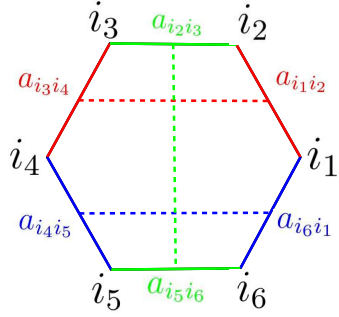}
\caption{\label{fig:pairingnoorientation} Diagram for a pairing arising in computing the $6$\textsuperscript{th} moment.}
\end{center}
\end{figure}

Note that this is dual to the diagrams for pairings that appear in \cite{HM, MMS}, in which the $a_{i_s i_{s+1}}$ are represented as vertices. For more on such an identification and its application in determining moments for random matrix ensembles, see \cite{Fo} (Section 1.6) and \cite{Zv}.

%ASSUME m|N

If $a_{i_s i_{s+1}}$ and $a_{i_t i_{t+1}}$ are paired, we have either
\begin{itemize}
\item
$i_{s+1}-i_s \con i_{t+1}-i_t \Mod{N} \mbox{ and } i_s \con i_t \Mod{m}, \mbox{ or}$
\item
$i_{s+1}-i_s \con -(i_{t+1}-i_t) \Mod{N} \mbox{ and } i_s \con i_{t+1} \Mod{m}$.
\end{itemize}

We think of these two cases as pairing $(s,s+1)$ and $(t,t+1)$ with the same or opposite orientation, respectively.  For example, in Figure \ref{fig:twohexagons} the hexagon on the left has all edges paired in opposite orientation, and the one on the right has all but the red edges paired in opposite orientation.

\begin{figure}
\begin{center}
\includegraphics[scale=.34]{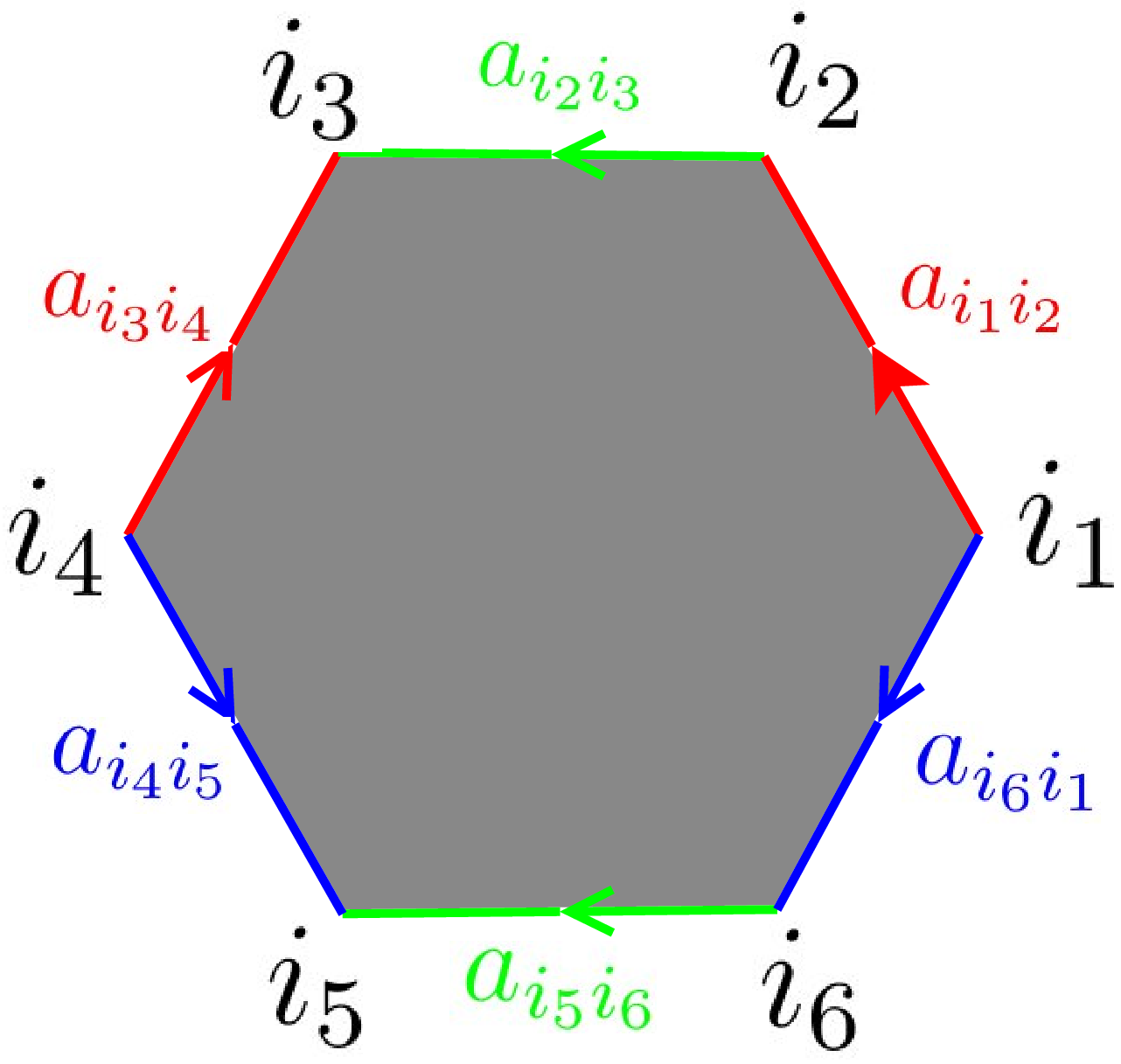}
\includegraphics[scale=.34]{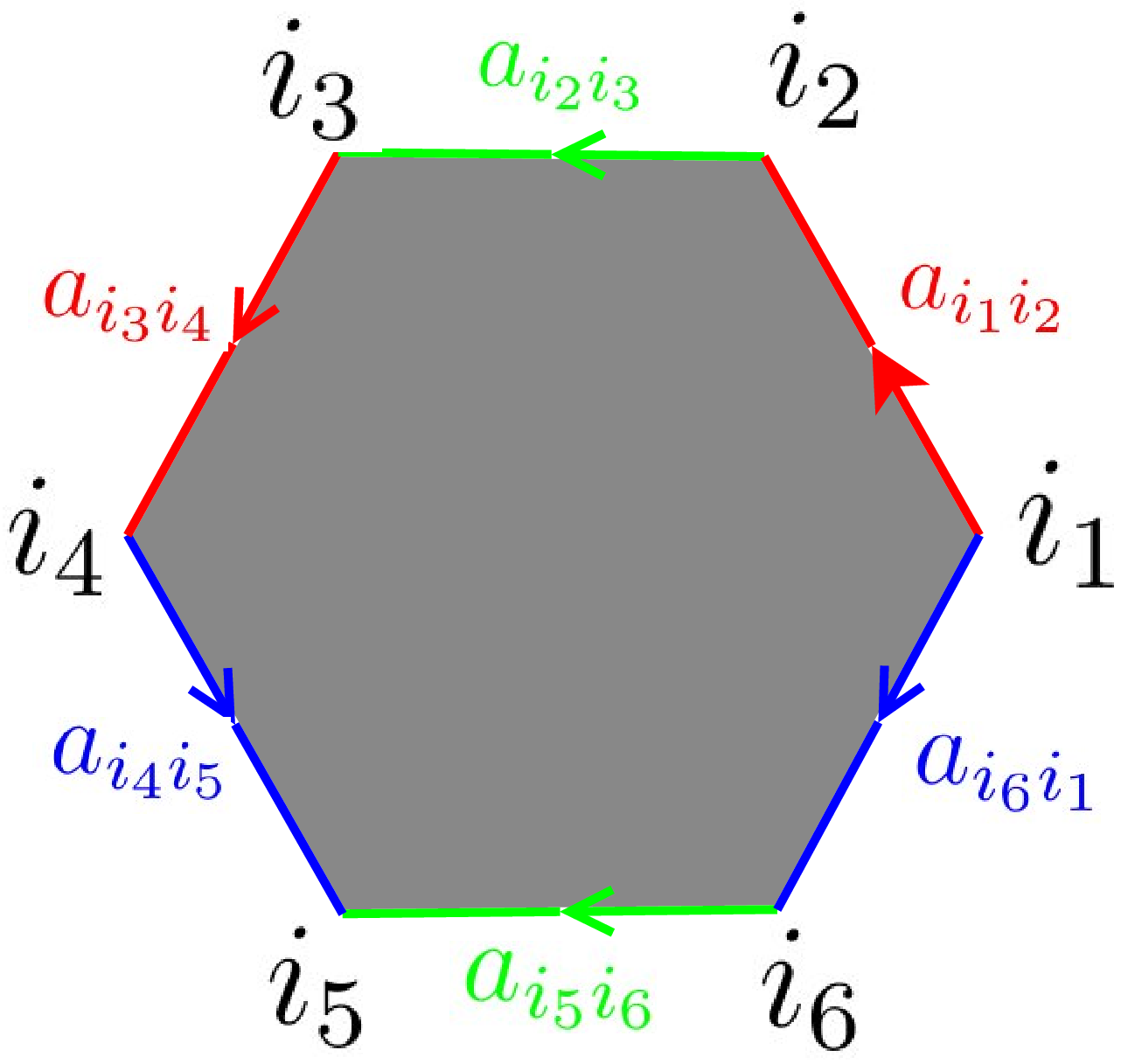}
\caption{\label{fig:twohexagons} Some possible orientations of paired edges for the $6$-gon.}
\end{center}
\end{figure}

We now dramatically reduce the number of pairings we must consider by showing that the only pairings which contribute in the limit are those in which all edges are paired with opposite orientation.  Topologically, these are exactly the pairings which give orientable surfaces \cite{Hat,HarZa}.  This result and its proof is a minor modification of their analogs in the Toeplitz and palindromic Toeplitz cases \cite{HM, MMS, JMP}.

\begin{lem}\label{lem: matchings}
%The only pairings which contribute in the limit are those in which all edges are paired with opposite orientation. ...
Consider a pairing $\s$ with orientations $\varepsilon_j$.  If any $\varepsilon_j$ is equal to $1$, then the pairing contributes $O_k(\nicefrac{1}{N})$.
\end{lem}
\begin{proof}
The size of the contribution is equal to the number of solutions to the $k$ equations
\begin{equation}\label{Nequations}
i_{s+1}-i_s\ \con\ \varepsilon_j(i_{\s(s)+1}-i_{\s(s)}) \Mod{N},
\end{equation}
as well as some $\Mod{m}$ equations, divided by $N^{k+1}$.
We temporarily ignore the $\Mod{m}$ constraints and bound the contribution from above by the number of solutions to the $\Mod{N}$ equations over $N^{k+1}$.  Because the $i_s$ are restricted to the values $1,2,\dots,N$, we can consider them as elements of $\Z/N\Z$, and we now notate the $\Mod{N}$ congruences with equality.

The pairing puts the numbers $1,2,\dots,2k$ into $k$ equivalence classes of size two; arbitrarily order the equivalence classes and pick an element from each to call $s_j$, naming the other element $t_j = \s(s_j)$.

Our $\Z/N\Z$ equations now look like
\begin{equation}
i_{s_j+1}-i_{s_j}\ =\ \varepsilon_j(i_{t_j+1}-i_{t_j}) \bmod N.
\end{equation}

Defining
\bea
x_j &\ :=\ & i_{s_j+1}-i_{s_j}\nn\\
y_j &\ := \ & i_{t_j+1}-i_{t_j},
\eea
our equations now look like $x_j = \varepsilon_j y_j$. Thus
\bea
0 &\ = \ & \sum_{s=1}^{2k} i_{s+1} - i_s \ = \
 \sum_{j=1}^{k} x_j + \sum_{j=1}^{k} y_j\nn \ = \ \sum_{j=1}^{k} (\varepsilon_j + 1)y_j.
\eea
If any one of the $\varepsilon_j = 1$, this gives a nontrivial relation among the $y_j$, and we lose a degree of freedom.  We may choose $k-1$ of the $y_j$ freely (in $\Z/N\Z$), and we are left with $1$ or possibly $2$ choices for the remaining $y_j$ (depending on the parity of $N$).  The $x_j$ are now determined as well, so $i_{s+1}-i_s$ is now determined for every $s$.  If we choose $i_1$ freely, this now determines all the $\ds i_s = i_1 + \sum_{s^\p < s} (i_{s^\p+1}-i_{s^\p})$.  Thus, we have at most $N^{k-1} \cdot 2 \cdot N = 2 N^k$ solutions to (\ref{Nequations}).  So the contribution from a pairing with a positive sign is at most $O_k(2 \nicefrac{N^k}{N^{k+1}}) = O_k(\nicefrac{1}{N})$. (The reason for the big-Oh constant depending on $k$ is that if some of the different pairs have the same value, we might not have $k$ copies of the second moment but instead maybe four second moments and two eighth moments; however, the contribution is trivially bounded by ${\max}_{1 \le \ell \le k} (1+m_{2\ell})^k$, where $m_{2\ell}$ is the $2\ell$\textsuperscript{th} moment of $p$.)
\end{proof}

Thus we have
\be
M_{2k;m}(N)\ =\ \sum_{\sigma} w(\sigma) N^{-(k+1)} + O_k\left(\frac{1}{N}\right),
\ee
where $w(\s)$ denotes the number of solutions to
\begin{equation}
i_{j+1}-i_j\ \con\ -(i_{\s{j}+1}-i_{\s{j}})\bmod N
\end{equation}
and
\be\label{mmm}
i_j\ \con\ i_{\s(j)+1},\ \ \  i_{j+1}\ \con\ i_{\s(j)}\bmod m
\ee
(the second $\Mod{m}$ constraint is redundant). We discuss how to evaluate this moment in closed form, culminating in Lemma \ref{lem:formulaforevenmomentsepsilon}.

%Now I'll calculate the contribution from each pairing with all negative signs.  Consider such a pairing $\s$.  We have equations of the form

We now consider a given pairing as a topological identification (see \cite{Hat} for an exposition of the standard theory); this is the crux of our argument. Specifically, consider a $2k$-gon with the interior filled in (homeomorphic to the disk), and identify the paired edges with opposite orientation.  Under the identification, some vertices are identified; let $v$ denote the number of vertices in the quotient.

Consider the $(\Z/N\Z)$-submodule $\mathcal{A}$ of $(\Z/N\Z)^{2k}$ in which the $\Mod{N}$ constraints hold.  We have $\mathcal{A}$ is isomorphic to $(\Z/N\Z)^{k+1}$.  Specifically, we may freely choose the value of exactly half of the differences $i_{s+1}-i_s$, and then the rest are determined.  Because all the pairings are opposite orientation, these ``differences'' sum to zero, so they are actually realizable as differences.  Now choose $i_1$ freely, and the rest of the $\ds i_s = i_1 + \sum_{s' < s} (i_{s'+1}-i_{s'})$ are determined.

Let $\bar{\mathcal{A}}$ denote the quotient of $\mathcal{A}$ in which everything is reduced modulo $m$, and consider the $(\Z/m\Z)$-submodule $B \subseteq \bar{\mathcal{A}}$ in which the modulo $m$ constraints hold.  By (\ref{mmm}), we can see that the labels at two vertices of our $2k$-gon are forced to be congruent $\Mod{m}$ if and only if the vertices are identified in the quotient, and these are all the $\Mod{m}$ constraints.  In other words, $\mathcal{B}$ is isomorphic to $(\Z/m\Z)^v$.  An element of $\mathcal{A}$ for which the $\Mod{m}$ constraints also hold is exactly one in the preimage of $\mathcal{B}$.  We have $m^v$ choices for an element in $\mathcal{B}$, and there are $(\nicefrac{N}{m})^{k+1}$ ways to lift such an element to an element of $A$ in its fiber.
Thus, the equations have a total of $m^v (\nicefrac{N}{m})^{k+1} = m^{-(k+1-v)} N^{k+1}$, so the pairing has a contribution of $m^{-(k+1-v)}$.

% I changed the CW-complex back to X because \mathcal{X} looks too much like \chi.  If you don't want it to be X for some reason, go ahead and call it something else, such as Y, S, M, or a \mathcal version of one of these.

Let $X$ be the 2-dimensional cell complex described by the pairing $\s$ of the edges of the $2k$-gon.  Because all edges were paired in the reverse direction, $X$ is an orientable surface.  After identifications, the complex we've described has 1 face, $k$ edges, and, say, $v$ vertices.  If we denote by $g$ the genus of the surface, we obtain two expressions for the Euler characteristic of $X$.  By the standard (homological) definition of Euler characteristic, we have $\chi(X) = 1-k+v$.  On the other hand, for a genus $g$ surface $X$, $\chi(X) = 2-2g$ \cite{Hat}.  Equating and rearranging,
\begin{eqnarray}
2g\ =\ k+1-v.
\end{eqnarray}
Thus the pairing $\s$ contributes $m^{-2g}$, and we have shown

\begin{lem}\label{lem:formulaforevenmomentsepsilon} For the ensemble of symmetric $m$-block circulant matrices,
\be
M_{2k;m}(N) \ =\ \sum_g \varepsilon_g(k) m^{-2g} + O_k\left(\frac{1}{N}\right),
\ee where $\varepsilon_g(k)$ denote the number of pairings of the edges of a $2k$-gon which give rise to a genus $g$ surface.
\end{lem}

%%%%%%%%%%%%%%%%%%%%%%%%%%%%%%%%%%%%%%%%%%%%%%%%%%%%%%%%%%%%%%%%%%%%%%%%%%%%%%%%%%%%%%%%%%%%%%%%%%%%%%%%%%%%
%%%%%%%%%%%%%%%%%%%%%%%%%%%%%%%%%%%%%%%%%%%%%%%%%%%%%%%%%%%%%%%%%%%%%%%%%%%%%%%%%%%%%%%%%%%%%%%%%%%%%%%%%%%%

\section{Determining the Limiting Spectral Measures}\label{sec:determiningevenmoments}
%\begin{thm}[The Density Function]
%The limiting spectral density function $f_m(x)$ of the real symmetric period $m$-block circulant ensemble is given by the formula
%
%\be f_m (x) = \frac{e^{-mx^2/2}}{\sqrt{2\pi m}} \sum_{l=1}^{m} \sum_{s=0}^{l-1} \ncr{m}{l}\frac{(2s-1)!!}{(l-1)!}   \ncr{2(l-1)}{2s} \left(m x^2 \right)^{l-1-s} (-1)^{s}.
%\ee
%\end{thm}

We prove parts (1) and (2) of Theorem \ref{thm:mainexpformulasconvergence}. Specifically, we derive the density formula for the limiting spectral density of symmetric $m$-block circulant matrices. We show that, if $m$ grows at any rate with $N$, then the limiting spectral density is the semi-circle for both the symmetric $m$-block circulant and Toeplitz ensembles.

\subsection{The Limiting Spectral Measure of the Symmetric $m$-Block Circulant Ensemble}

\begin{proof}[Proof of Theorem \ref{thm:mainexpformulasconvergence}(1)]

By deriving an explicit formula, we show that the limiting spectral density function $f_m$ of the real symmetric $m$-block circulant ensemble is equal to the spectral density function of the $m \times m$ GUE.

%We show that the limiting spectral density function $f_m(x)$ of the real symmetric $m$-block circulant ensemble is given by the formula
%\be
%f_m(x)\ =\ \frac{ e^{-mx^2/2}}{\sqrt{2\pi m}} \sum_{r=0}^{m-1} \frac{1}{(2r)!}
%\left( \sum_{s=0}^{m-r} {m \choose r+s+1} \frac{(2r+2s)!}{(r+s)! s!} \left(-\foh \right)^s \right)
%(m x^2)^r.
%\ee

From Lemma \ref{lem:formulaforevenmomentsepsilon}, the $N\to\infty$ limit of the average $2k$\textsuperscript{th} moment equals \be
M_{2k;m}\ = \ \sum_{g=0}^{\lfloor \nicefrac{k}{2} \rfloor}\varepsilon_g(k) m^{-2g}, \ee with $\varepsilon_g(k)$ the number of pairings of the edges of a $2k$-gon giving rise to a genus $g$ surface. Harer and Zagier \cite{HarZa} give formulas for the $\varepsilon_g(k)$.  They prove
\be\label{eq:tanhvarepsilon} \varepsilon_g(k)\ =\ \frac{(2k)!}{(k+1)!(k-2g)!} \times \left(\mbox{coefficient of } x^{2g} \mbox{ in } \left(\frac{\nicefrac{x}{2}}{\tanh(\nicefrac{x}{2})} \right)^{k+1}\right) \ee
and
\be \sum_{g=0}^{\lfloor \nicefrac{k}{2} \rfloor} \varepsilon_g (k) r^{k+1-2g}\ =\ (2k-1)!! \ c(k,r), \ee
where
\be 1+2 \sum_{k=0}^{\infty} c(k,r) x^{k+1}\ =\ \left(\frac{1+x}{1-x}\right)^r. \ee
Thus, we may write
\be
M_{2k;m}\ =\ m^{-(k+1)}(2k-1)!!\,c(k,m).
\ee

We construct the characteristic function\footnote{\label{foot:charfnfootnote} The characteristic function is $\phi_m(t) = \E[e^{itX_m}] = \int_{-\infty}^\infty f_m(x) e^{itx}dx$. This is the inverse Fourier transform of $f_m$.} of the limiting spectral distribution. Let $X_m$ be a random variable with density $f_m$. Then (remembering the odd moments vanish)
\bea
\phi_m(t) &\ =\ & \E[e^{itX_m}] \ = \ \sum_{\ell=0}^\infty \frac{(it)^\ell M_{\ell;m}}{\ell!} \nonumber\\ &\ = \ &  \sum_{k=0}^{\infty} \frac{(i t)^{2k} M_{2k;m}}{(2k)!}\nn \\
       &\ =\ & \sum_{k=0}^{\infty} \frac{1}{(2k)!}m^{-(k+1)}(2k-1)!!\,c(k,m) (-t^2)^k.
\eea

In order to obtain a closed form expression, we rewrite the characteristic function as
\be
\phi_m(t)\ =\ \frac{1}{m} \sum_{k=0}^{\infty} c(k,m) \frac{1}{k!} \left(\frac{-t^2}{2m}\right)^k,
\ee
using $(2k-1)!!=\frac{(2k)!}{2^k k!}$. The reason for this is that we can interpret the above as a certain coefficient in the convolution of two known generating functions, which can be isolated by a contour integral. Specifically, consider the two functions
\be
F(y)\ := \ \frac{1}{2y} \left(\left(\frac{1+y}{1-y}\right)^m -1 \right)\ =\ \sum_{k=0}^{\infty} c(k,m) y^{k} \ \ \ {\rm and} \ \ \
G(y)\ :=\ e^{y}\ =\ \sum_{k=0}^{\infty} \frac{ y^{k}}{k!}. \ee
Note that $\phi_m(t)$ is the function whose power series is the sum of the products of the $k$\textsuperscript{th} coefficients of $G(-\nicefrac{y^2}{2m})$ (which is related to the exponential distribution) and $F(y)$ (which is related to the generating function of the $\varepsilon_g(k)$). Thus, we may use a multiplicative convolution to find a formula for the sum.  By Cauchy's residue theorem, integrating $F(z^{-1}) G(-\nicefrac{t^2 z}{2m}) z^{-1}$ over the circle of radius $2$ yields
\be \phi_m(t)\ =\ \frac{1}{2\pi i m} \oint_{\left|z\right|=2} F(z^{-1})G\left(-\frac{t^2 z}{2m} \right)\frac{dz}{z}, \ee
since the constant term in the expansion of $F(z^{-1}) G(-\nicefrac{t^2 z}{2m})$ is exactly the sum of the products of coefficients for which the powers of $y$ in $F(y)$ and $G(y)$ are the same.\footnote{All functions are meromorphic in the region with finitely many poles; thus the contour integral yields the sum of the residues. See for example \cite{SS2}.} We are integrating along the circle of radius $2$ instead of the unit circle to have the pole inside the circle and not on it. Thus
\bea \phi_m(t) & \ = \ & \frac{1}{2\pi i m} \oint_{\left|z\right|=2} \frac{1}{2z^{-1}} \left(\left(\frac{1+z^{-1}}{1-z^{-1}}\right)^m -1 \right)e^{-\nicefrac{t^2 z}{2m}} \frac{dz}{z} \nonumber
\\ & = & \frac{1}{4\pi i m} \oint_{\left|z\right|=2} \left(\left(\frac{z+1}{z-1}\right)^m -1 \right)e^{-\nicefrac{t^2z}{2m}} dz \nonumber
\\ & = & \frac{e^{-\nicefrac{t^2}{2m}}}{4\pi i m} \oint_{\left|z\right|=2} \left(\left(1+\frac{2}{z-1}\right)^m -1 \right)e^{-\nicefrac{t^2(z-1)}{2m}} dz \nonumber
\\ & = & \frac{e^{-\nicefrac{t^2}{2m}}}{4\pi i m} \oint_{\left|z\right|=2}  \sum_{l=0}^{m}\ncr{m}{l}\left(\frac{2}{z-1}\right)^l \sum_{s=0}^{\infty}\frac{1}{s!}\left(\frac{-t^2}{2m}\right)^s(z-1)^s dz  \nonumber \\ & & \ \ \ \ \ - \frac{e^{-\nicefrac{t^2}{2m}}}{4\pi i m} \oint_{\left|z\right|=2} e^{-\nicefrac{t^2(z-1)}{2m}} dz.
\eea

By Cauchy's Residue Theorem the second integral vanishes and the only surviving terms in the first integral are when $l-s=1$, whose coefficient is the residue. Thus
\bea\label{eq:goodexpansioncharfnm}
\phi_m(t) &\ = \ & \frac{e^{-\nicefrac{t^2}{2m}}}{2 m} \sum_{l=1}^{m} \ncr{m}{l} 2^l \frac{1}{(l-1)!}\left(\frac{-t^2}{2m}\right)^{l-1} \nn \\
        &=& \frac{1}{m}\ e^{-\nicefrac{t^2}{2m}} \sum_{l=1}^m {m \choose l} \frac{1}{(l-1)!}\left(\frac{-t^2}{m}\right)^{l-1}\ = \ \frac{1}{m}\ e^{-\nicefrac{t^2}{2m}} L_{m-1}^{(1)}\left(\nicefrac{t^2}{m}\right),
\eea
which equals the spectral density function of the $m\times m$ GUE (see \cite{Led}).

% = \sum_{k=0}^{\infty} \frac{( -2t^2 / m)^{k}}{k!}
%We can rearrange the second expression and sum over all $0\leq k \leq \infty$

As the density and the characteristic function are a Fourier transform pair, each can be recovered from the other through either the Fourier or the inverse Fourier transform (see for example \cite{SS1,SS2}). Since the characteristic function is given by
\be\label{eq:fourdefphifm} \phi_m(t) \ = \ \E[e^{itX_m}] \ =\ \int_{-\infty}^{\infty} e^{itx} f_m(x) \,dx \ee (where $X_m$ is a random variable with density $f_m$), the density is regained by the relation
\be f_m(x)\ =\ \widehat{\phi_m}(x)\ =\ \frac{1}{2\pi} \int_{-\infty}^{\infty} e^{-itx}\phi_m(t)\,dt. \ee
Taking the Fourier transform of the characteristic function $\phi_m(t)$, and interchanging the sum and the integral, we get
\bea \label{eq:den1}
f_m(x) &\ = \ &  \frac{1}{2\pi} \int_{-\infty}^{\infty}\frac{e^{-\nicefrac{t^2}{2m}}}{m} \sum_{l=1}^{m} \ncr{m}{l}  \frac{1}{(l-1)!}\left(\frac{-t^2}{m}\right)^{l-1} e^{-itx}dt \nonumber
\\ &\ =\ & -\frac{1}{2\pi} \sum_{l=1}^{m} \ncr{m}{l}\frac{1}{(l-1)!} \left(-m\right)^{-l} \int_{-\infty}^{\infty} t^{2(l-1)} e^{-\nicefrac{t^2}{2m}} e^{-itx} dt \nonumber\\ &\ =\ & -\frac{1}{2\pi} \sum_{l=1}^{m} \ncr{m}{l}\frac{1}{(l-1)!} \left(-m\right)^{-l} I_m.
\eea
Completing the square in the integrand of $I_m$, we obtain
\bea I_m   &\ =\ & e^{-\nicefrac{mx^2}{2}} \int_{-\infty}^{\infty} t^{2(l-1)} \exp\left(-\foh \left(\frac{t}{\sqrt{m}} +i\sqrt{m}x\right)^2\right) dt.
\eea
Changing variables by $y=\frac{1}{\sqrt{m}}t +i\sqrt{m}x $, $dy= \frac{1}{\sqrt{m}} dt$, we find $I_m$ equals
\bea I_m &\ =\ & e^{-\nicefrac{mx^2}{2}} \int_{-\infty}^{\infty} \left( y-i\sqrt{m}x \right)^{2(l-1)} \left(\sqrt{m}\right)^{2(l-1)}  e^{-\nicefrac{y^2}{2}} \sqrt{m} dy \nonumber\\ & = & e^{-\nicefrac{mx^2}{2}} m^{l-\foh} \sum_{s=0}^{2(l-1)} \ncr{2(l-1)}{s} \left(-i\sqrt{m}x \right)^{2(l-1)-s} \int_{-\infty}^{\infty}  y^s e^{-\nicefrac{y^2}{2}} dy.
\eea
The integral above is the $s$\textsuperscript{th} moment of the Gaussian, and is $\sqrt{2\pi} (s-1)!!$ for even $s$ and $0$ for odd $s$. Since the odd $s$ terms vanish, we replace the variable $s$ with $2s$ and sum over $0\leq s \leq (l-1)$. We find
\bea I_m \ = \ \sqrt{2\pi} e^{-\nicefrac{mx^2}{2}} m^{l-\foh} \sum_{s=0}^{l-1} \ncr{2(l-1)}{2s} \left(-m x^2 \right)^{l-1-s} (2s-1)!!. \nonumber \\ \eea
Substituting this expression for $I_m$ into \eqref{eq:den1} and making the change of variables $r = l-1-s$,  we find that the density is
%\be f_m (x)\ =\ \frac{e^{-mx^2/2}}{\sqrt{2\pi m}} \sum_{l=1}^{m} \sum_{s=0}^{l-1} \ncr{m}{l}\frac{(2s-1)!!}{(l-1)!}   \ncr{2(l-1)}{2s} \left(m x^2 \right)^{l-1-s} (-1)^{s}.
%\ee
%\be \frac{-1}{2\pi} \sum_{l=1}^{m} \ncr{m}{l}\frac{1}{(l-1)!} \left(-m\right)^{-l} \sqrt{\pi} e^{-mx^2/2} m^{l-\foh} \sum_{s=0}^{l-1} \ncr{2(l-1)}{2s} \left(-m x^2 \right)^{l-1-s} (2s-1)!! \ee
\be
f_m(x)\ =\ \frac{ e^{-\nicefrac{mx^2}{2}}}{\sqrt{2\pi m}} \sum_{r=0}^{m-1} \frac{1}{(2r)!}
\left( \sum_{s=0}^{m-r} {m \choose r+s+1} \frac{(2r+2s)!}{(r+s)! s!} \left(-\foh \right)^s \right)
(m x^2)^r.
\ee This completes the proof of Theorem \ref{thm:mainexpformulasconvergence}(1).
\end{proof}

%\begin{rek} The terrific fits of our numerical investigations (see Figures \ref{figure:hist2hist3} and \ref{figure:hist4hist8}) strongly suggest that we have not made any significant algebra errors.
%\end{rek}

\subsection{The $m\to\infty$ Limit and the Semi-Circle}

Before proving Theorem \ref{thm:mainexpformulasconvergence}(2), we first derive expressions for the limits of the average moments of the symmetric $m$-block Toeplitz ensemble. We sketch the argument. Though the analysis is similar to its circulant cousin, it presents more difficult combinatorics.  Because diagonals do not ``wrap around'', certain diagonals are better to be on than others.  Consequently, the Diophantine obstructions of \cite{HM} are present. The problems are the matchings with ``crossings'', or, in topological language, those matchings which give rise to tori with genus $g \geq 1$ as opposed to spheres with $g=0$. For a detailed analysis of the Diophantine obstructions and how the added circulant structure fixes them, see \cite{HM} and \cite{MMS}. Fortunately, it is easy to show that the contributions to the $2k$\textsuperscript{th} moment of the symmetric $m$-block Toeplitz distribution from the non-crossing (i.e, the spherical matchings or, in the language of \cite{BanBo}, the Catalan words) are unhindered by Diophantine obstructions and thus contribute fully. The number of these matchings is $C_k$, which is the $k$\textsuperscript{{\rm th}} Catalan number $\frac1{k+1}\ncr{2k}{k}$ as well as the $2k$\textsuperscript{th} moment of the Wigner density \be \twocase{f_{\rm Wig}(x) \ = \ }{\frac1{2\pi}\sqrt{1-\left(\frac{x}{2}\right)^2}}{if $|x| \le 2$}{0}{otherwise.} \ee Note that with this normalization have a semi-ellipse and not a semi-circle; to obtain the semi-circle, we normalize the eigenvalues by $2\sqrt{N}$ and not $\sqrt{N}$. As the other matchings contribute zero in the limit, we obtain convergence to the Wigner semi-circle as $m \to \infty$. We now prove the above assertions.

\begin{lem}\label{lem:limitsmomentssymmperiodmToep}
The limit of the average of the $2k$\textsuperscript{th} moment of the symmetric $m$-block Toeplitz ensemble equals
\be M_{2k;m}\ =\ C_k + \sum_{g=1}^{\lfloor \nicefrac{k}{2} \rfloor} d(k,g) m^{-2g}, \ee
where $C_k$ is the $k$\textsuperscript{{\rm th}} Catalan number and $d(k,g) \in [0,1]$ are constants corresponding to the total  contributions from the genus $g$ pairings for the $2k$\textsuperscript{{\rm th}} moment.
\end{lem}

\begin{proof}
For the symmetric $m$-block Toeplitz ensemble, the analysis in \S\ref{sec:trace} applies almost exactly. In the condition for $a_{ij} = a_{i'j'}$, equality replaces congruence modulo $N$.
\begin{itemize}
\item
$j-i = j^\p-i^\p \mbox{ and } i \con i^\p \Mod{m}, \mbox{ or}$
\item
$j-i = -(j^\p-i^\p) \mbox{ and } i \con j^\p \Mod{m}.$
\end{itemize}

These constraints are more restrictive, so we again obtain $2^{n-l} N^{l+1}$ as an upper bound on the number of solutions. Following the previous argument, the odd moments are $M_{2k+1;m}(N)=O_k(\nicefrac{1}{\sqrt{N}})$, and the even moments are
\be M_{2k;m}(N)\ =\ \frac{1}{N^{k+1}} \sum_{\s} \eta(\s) + O_k\left(\frac{1}{N}\right), \ee
where $ \eta(\s) $ is the number of solutions to the Diophantine equations arising from the pairings $\s$ on $\{(1,2),(2,3),\dots,(2k,1)\}$ of the indices. Thus the odd moments vanish in the limit. Moreover, the only matchings that contribute are the ones with negative signs. To see this fact, one can follow the proof of Lemma \ref{lem: matchings}, except working in $\Z$ instead of $\Z/N\Z$.

While it is known that most matchings for the real symmetric Toeplitz ensemble do not contribute fully, a general expression for the size of the contributions is unknown, though there are expressions for these in terms of volumes of Eulerian solids (see \cite{BDJ}) or obstructions to Diophantine equations (see \cite{HM}). These expressions imply that each matching contributes at most 1. We introduce constants to denote their contribution (this corresponds to the $m=1$ case). This allows us to handle the real symmetric $m$-block Toeplitz ensemble, and (arguing as in the proof of Lemma \ref{lem:formulaforevenmomentsepsilon}), write the limit of the average of the $2k$\textsuperscript{{\rm th}} moments as %\textbf{(IS IT TRIVIAL / OBVIOUS THAT $d(k,g)$ WILL NOT DEPEND ON $m$, IN OTHER WORDS, THAT PULLING OUT $m^{-2g}$ REMOVES ALL THE $m$-DEPENDENCE?)}
% It follow from the proof of 3.2.  I can add more explanation, but it's tedious and not very interesting.
\be M_{2k;m}\ =\ \sum_{g=0}^{\lfloor \nicefrac{k}{2} \rfloor} d(k,g) m^{-2g}. \ee
Here $d(k,g)$ is the constant corresponding to the contributions of the genus $g$ matchings. All that is left is to show that $d(k,0)$, the contributions from the non-crossing or spherical matchings, is the Catalan number $C_k$.

We know that the number of non-crossing matchings of $2k$ objects into $k$ pairs is the Catalan number $C_k$. This is well-known in the literature. Alternatively, we know the number of non-crossing matchings are $\varepsilon_0(k)$, as these are the ones that give the genus 0 sphere. The claim follows immediately from \eqref{eq:tanhvarepsilon} by taking the constant term (as $g=0$) and noting $\tanh(\frac{x}{2}) = \frac{x}{2} - \frac{x^3}{24} + \cdots$. We are thus reduced to proving that, even with the mod $m$ periodicity, each of these pairings still contributes 1.

One way of doing this is by induction on matchings. Consider a non-crossing configuration of contributing matchings for the $2k$\textsuperscript{th} moment. Consider an arbitrary matching in the configuration, and denote the matching by $\ga_1$.  The matching corresponds to an equation $i_s-i_{s+1}=i_{t+1}-i_t$. If the matching is adjacent, meaning $s=t+1$, then $i_{t+1}$ is free and $i_t=i_{t+2}$, and there is no ``penalty'' (i.e., a decrease in the contribution) from the $\Mod{m}$ condition. We call this having the ends of a matching ``tied'' (note that adjacent matchings always tie their ends). Otherwise, note that since we are looking at even moments, there are an even number of indices. Thus, to either side of the matching $\ga_1$ there can only be an even number of indices matched between themselves, since otherwise some matching would be crossing over $\ga_1$. Thus, to either side, we are reduced to the non-crossing configurations for a lesser moment. By induction, these two sub-configurations are tied, and then trivially tie with our initial matched pair. As at each step there were no obstructions on the indices, this matching contributes fully, completing the proof.
\end{proof}

Our claims about convergence to semi-circular behavior now follow immediately.

\begin{proof}[Proof of Theorem \ref{thm:mainexpformulasconvergence}(2)] It is trivial to show that the symmetric $m$-block circulant ensemble has its limiting spectral distribution converge to the semi-ellipse as $m\to\infty$ because we have an explicit formula for its moments. From Lemma \eqref{lem:formulaforevenmomentsepsilon}, we see that \be \lim_{m\to\infty} M_{2k;m}(N) \ = \ \lim_{m\to\infty} \sum_{g \le \nicefrac{k}{2}} \frac{\varepsilon_g(k)}{m^{2g}} \ = \ \varepsilon_0(k),\ee which in the proof of Lemma \ref{lem:limitsmomentssymmperiodmToep} we saw equals the Catalan number $C_k$.

We now turn to the symmetric $m$-block Toeplitz case. The proof proceeds similarly. From Lemma \ref{lem:limitsmomentssymmperiodmToep} we have \be \lim_{m\to\infty} M_{2k;m}\ =\ \lim_{m\to\infty} \left(C_k + \sum_{g \le \nicefrac{k}{2}} \frac{d(k,g)}{m^{2g}}\right) \ = \ C_k, \ee
completing the proof. \end{proof}

%%%%%%%%%%%%%%%%%%%%%%%%%%%%%%%%%%%%%%%%%%%%%%%%%%%%%%%%%%%%%%%%%%%%%%%%%%%%%%%%%%%%%%%%%%%%%%%%%%%%%%%%%%%%
%%%%%%%%%%%%%%%%%%%%%%%%%%%%%%%%%%%%%%%%%%%%%%%%%%%%%%%%%%%%%%%%%%%%%%%%%%%%%%%%%%%%%%%%%%%%%%%%%%%%%%%%%%%%

\section{Convergence of the Limiting Spectral Measures}\label{sec:convergence}

We investigate several types of convergence.

\ben

\item (Almost sure convergence) For each $k$, $X_{k;m,N} \to X_{k,m}$
almost surely if \be \pp_{m}\left(\{A \in \Omega_{m}: X_{k;m,N}(A) \to
X_{k,m}(A)\ {\rm as}\ N\to\infty\}\right) \ =  \ 1; \ee

\item (Convergence in probability) For each $k$, $X_{k;m,N} \to X_{k,m}$ in
probability if for all $\gep>0$, \be
\lim_{N\to\infty}\pp_{m}(|X_{k;m,N}(A) - X_{k,m}(A)|
> \gep) \ = \ 0;\ee

\item (Weak convergence) For each $k$, $X_{k;m,N} \to X_{k,m}$ weakly if
\be \pp_{m}(X_{k;m,N}(A) \le x) \ \to \ \pp(X_{k,m}(A) \le x) \ee as
$N\to\infty$ for all $x$ at which $F_{X_{k,m}}(x) := \pp(X_{k,m}(A) \le x)$ is
continuous.

\een

Alternate notations are to say either \emph{with probability 1} or \emph{strongly} for almost
sure convergence and \emph{in distribution} for weak convergence;
both almost sure convergence and convergence in probability imply
weak convergence. For our purposes we take $X_{k,m}$ as the random
variable which is identically $M_{k,m}$, the  limit of the average $m$\textsuperscript{{\rm th}} moment (i.e., $\lim_{N\to\infty} M_{k,m;N}$), which we show below exist and uniquely determine a probability distribution for our ensembles.

%\emph{We often write $M_{k;m}(N)$ for $M_{k;m,N}$ to emphasize that $k$ and $m$ are fixed and $N$ tends to infinity.}

We have proved the first two parts of Theorem \ref{thm:mainexpformulasconvergence}, which tells us that the limiting spectral measures exist and giving us, for the symmetric $m$-block circulant ensemble, a closed form expression for the density. We now prove the rest of the theorem, and determine the various types of convergence we have. We first prove the claimed uniform convergence of part (3), and then discuss the weak, in probability, and almost sure convergence of part (4).

We use characteristic functions and Fourier analysis to show uniform (and thus pointwise) convergence of the limiting spectral distribution of the symmetric $m$-block circulant  ensemble to the semi-ellipse distribution (remember it is an semi-ellipse and not a semi-circle due to our normalization). We note that this implies $L^p$ convergence for every $p$. The proof follows by showing the characteristic functions are close, and then the Fourier transform gives the densities are close.

\begin{proof}[Proof of Theorem \ref{thm:mainexpformulasconvergence}(3)] The density $f_m$ is the Fourier transform of $\phi_m$ (equivalently, $\phi_m$ is the characteristic function associated to the density $f_m$, where we have to be slightly careful to keep track of the normalization of the Fourier transform; see \eqref{eq:fourdefphifm}); similarly the Wigner distribution $f_{\rm Wig}(x)$ is the Fourier transform of $\phi$, where the Wigner distribution (a semi-ellipse in our case due to our normalizations) is \be \twocase{f_{\rm Wig}(x) \ = \ }{\frac{1}{\pi} \sqrt{1 - \left(\frac{x}{2}\right)^2}}{if $|x| \le 2$}{0}{otherwise.}\ee

As our densities are nice, we may use the Fourier inversion formula to evaluate the difference. We find for any $\epsilon>0$ that
\bea
|\widehat{\phi}_m(x) - \widehat{\phi}(x)| & \ = \ & \left|\frac1{2\pi}\int_{-\infty}^\infty \left(\phi_m(t) - \phi(t)\right) e^{-itx}dt\right| \nonumber\\
& \leq & \int_{-\infty}^{\infty} |\phi_m(t) - \phi(t)|\,dt \nonumber\\ & \ll & m^{-\nicefrac{2}{9}+\epsilon}, \eea where the bound for this integral is proved in Lemma \ref{lem:l1boundsphimphi} and follows from standard properties of Laguerre polynomials and Bessel functions. Thus, as $m \to \infty$, $f_m(x) = \widehat{\phi}_m(x)$ converges to  $f_{\rm Wig}(x)=\widehat{\phi}(x)$ for all $x \in \R$.  As the bound on the difference depends only on $m$ and not on $x$, the convergence is uniform.

We now show $L^p$ convergence. We have $L^\infty$ convergence because it is equivalent to a.e. uniform convergence.  For $1 \leq p < \infty$, we automatically have $L^p$ convergence as we have both $L^1$ convergence and the $L^\infty$ norm is bounded.
\end{proof}

%\bea \norm f_m - f_{\rm Wig} \norm_{L^p}^p &\ = \   & \int_{-2}^2 |f_m(x)-f_{\rm Wig}(x)|^pdx + \int_{\R \setminus [-2,2]} |f_m(x)-f_{\rm Wig}(x)|^pdx \nn \\ & \ \le \ & \sup_{x\in\R} |f_m(x)-f_{\rm Wig}(x)|^{p-1} \left[\int_{-2}^2 |f_m(x)-f_{\rm Wig}(x)|dx + \int_{\R \setminus [-2,2]} |f_m(x)|dx\right] \nn \\   & \ll &  \log^{-(p-1)/4} m \left[\int_{-2}^2 |f_m(x)-f_{\rm Wig}(x)|dx + \int_{\R \setminus [-2,2]} |f_m(x)|dx\right]; \eea thus the problem reduces to proving $L^1$ convergence. The first integral $\ll 4 \log^{-1/4}m$, as that is the length of the interval times the maximum difference between $f_m$ and $f_{\rm Wig}$. The second integral is how much probability the density $f_m$ has outside $[-2,2]$. Since that density is always within $C\log^{-1/4}m$ of $f_{\rm Wig}$ for some $C$ independent of $m$, and $f_{\rm Wig}$ has all of its probability in $[-2,2]$, the greatest probability $f_m$ could assign outside $[-2,2]$ is $4C\log^{-1/4}m$. Collecting the pieces, we find \bea \norm f_m - f_{\rm Wig} \norm_{L^p}^p  \ \ll   \ \log^{-1/4}m, \eea which tends to zero as $m\to\infty$.

\begin{proof}[Proof of Theorem \ref{thm:mainexpformulasconvergence}(4)]
The proofs of these statements follow almost immediately from the arguments in \cite{HM,MMS,JMP}, as those proofs relied on degree of freedom arguments. The additional structure imposed by the $\Mod{m}$ relations does not substantially affect those proofs (as can seen in the generalizations of the arguments from \cite{HM} to \cite{MMS} to \cite{JMP}).
\end{proof}

%%%%%%%%%%%%%%%%%%%%%%%%%%%%%%%%%%%%%%%%%%%%%%%%%%%%%%%%%%%%%%%%%%%%%%%%%%%%%%%%%%%%%%%%%%%%%%%%%%%%%%%%%%%%
%%%%%%%%%%%%%%%%%%%%%%%%%%%%%%%%%%%%%%%%%%%%%%%%%%%%%%%%%%%%%%%%%%%%%%%%%%%%%%%%%%%%%%%%%%%%%%%%%%%%%%%%%%%%

\section{Future Research}\label{sec:futurework}

We discuss some natural, additional questions which we hope to study in future work.

\subsection{Representation Theory}\label{sec:repn}

The $N \times N$ $m$-block circulant matrices form a semisimple algebra over $\R$.  This algebra may be decomposed into $N$ simple subalgebras of dimension $m^2$, all but one or two of which are isomorphic to $\M_m(\C)$.  One can show that, up to first order, this decomposition sends our measure on symmetric $m$-block circulant matrices to the $m \times m$ Gaussian Unitary Ensemble.  One may then give a more algebraic proof of our results and circumvent the combinatorics of pairings; combining the two proofs gives a new proof of the results of \cite{HarZa}. This approach will appear in a more general setting in an upcoming paper of Kopp.  The general result may be regarded as a central limit theorem for Artin-Wedderburn decomposition of finite-dimensional semisimple algebras

\subsection{Spacings}

Another interesting topic to explore is the normalized spacings between adjacent eigenvalues. For many years, one of the biggest conjectures in random matrix theory was that if the entries of a full, $N\times N$ real symmetric matrix were chosen from a nice density $p$ (say mean 0, variance 1, and finite higher moments), then as $N\to\infty$ the spacing between normalized eigenvalues converges to the scaling limit of the GOE, the Gaussian Orthogonal Ensemble (these matrices have entries chosen from Gaussians, with different variances depending on whether or not the element is on the main diagonal or not). After resisting attacks for decades, this conjecture was finally proved; see the work of Erd\H{o}s, Ramirez, Schlein, and Yau \cite{ERSY,ESY} and Tao and Vu \cite{TV1,TV2}.

While this universality of behavior for differences seems to hold, not just for these full ensembles, but also for thin ensembles such as $d$-regular graphs (see the numerical observations of Jakobson, (S. D.) Miller, Rivin and Rudnick \cite{JMRR}), we clearly do not expect to see GOE behavior for all thin families. A simple counterexample are diagonal matrices; as $N\to\infty$ the density of normalized eigenvalues will be whatever density the entries are drawn from, and the spacings between normalized eigenvalues will converge to the exponential. We also see this exponential behavior in other ensembles. It has numerically been observed in various Toeplitz ensembles (see \cite{HM,MMS}).

For the ensemble of symmetric circulant matrices, we cannot have strictly exponential behavior because all but $1$ or $2$ (depending on the parity of $\nicefrac{N}{m}$) of the eigenvalues occur with multiplicity two.  This can be seen from the explicit formula for the eigenvalues of a circulant matrix.  Thus, the limiting spacing density has a point of mass $\frac{1}{2}$ at $0$.  Nonetheless, the \emph{nonzero} spacings appear to be distributed exponentially; see Figure \ref{fig:normspacingsymmcirc}.

\begin{figure}
\begin{center}
\includegraphics[scale=.7]{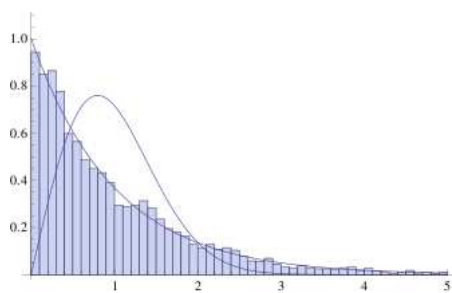}
\caption{\label{fig:normspacingsymmcirc} Density of nonzero spacings of the 10 central eigenvalues of 100 $1024 \times 1024$ symmetric circulant matrices, with independent entries picked i.i.d.r.v. from a Gaussian, normalized to have mean spacing 1.  Compared to exponential and GOE densities.
%\textbf{USE AUTOMATIC AND PROBABILITY DENSITY TO SCALE IT PROPERLY, AND INCLUDE THE STANDARD EXPONENTIAL; THUS YOU SHOULD NORMALIZE THE SPACINGS TO BE ON AVERAGE 1. ARE YOU LOOKING AT ALL THE SPACINGS, OR JUST A SUBSET IN THE BULK?}
}
\end{center}
\end{figure}

Similarly, for a symmetric $m$-block circulant matrix, all but $N-m$ or $N-m-1$ of the eigenvalues occur with multiplicity two.  The nonzero spacings appear to have the same exponential distribution (see Figure \ref{fig:normevsymmpermcirc}). %\textbf{(BE VERY CAREFUL -- DOES IT STILL LOOK EXPONENTIAL, OR MAYBE A LITTLE LESS SO. WE COULD DO SOME $\chi^2$ TESTS TO SEE IF THE CONVERGENCE IS LESSENING WHEN WE INCREASE $m$ TO 2.)}
This is somewhat surprising, given that the eigenvalue density varies with $m$ and converges to the semi-circle as $m \to \infty$.  While we see new eigenvalue densities for $m$ constant, numerics suggest that we'll see new spacing densities for $\nicefrac{N}{m}$ constant. %\textbf{NEED MORE INFO / SIMS TO SUPPORT CLAIM}

\begin{figure}
\begin{center}
\includegraphics[scale=.7]{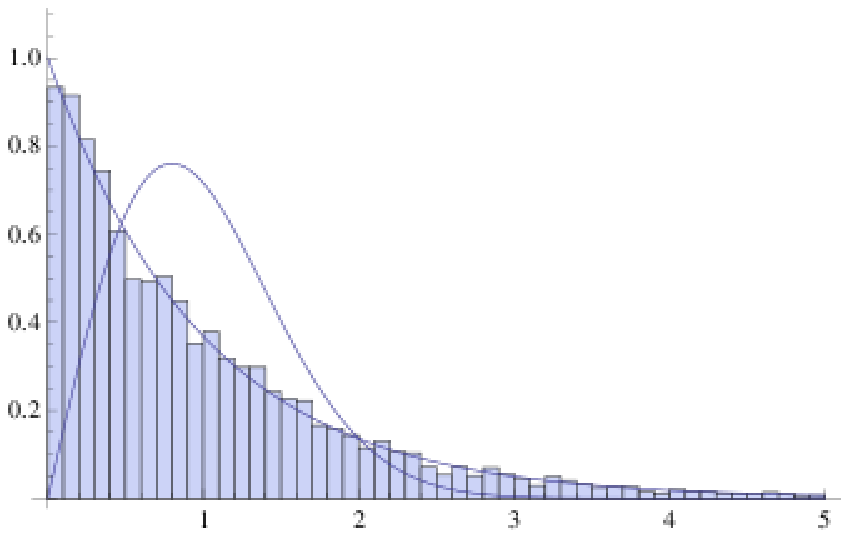}
\includegraphics[scale=.7]{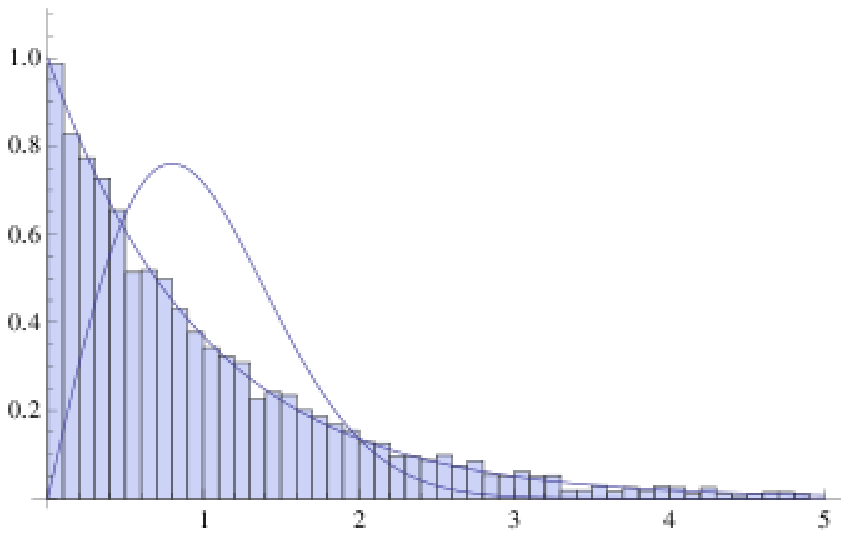}
\includegraphics[scale=.7]{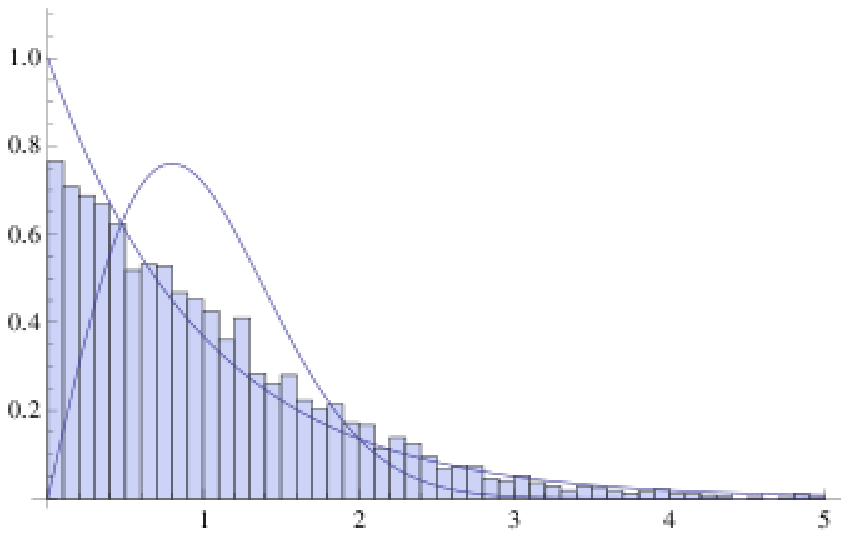}
\includegraphics[scale=.7]{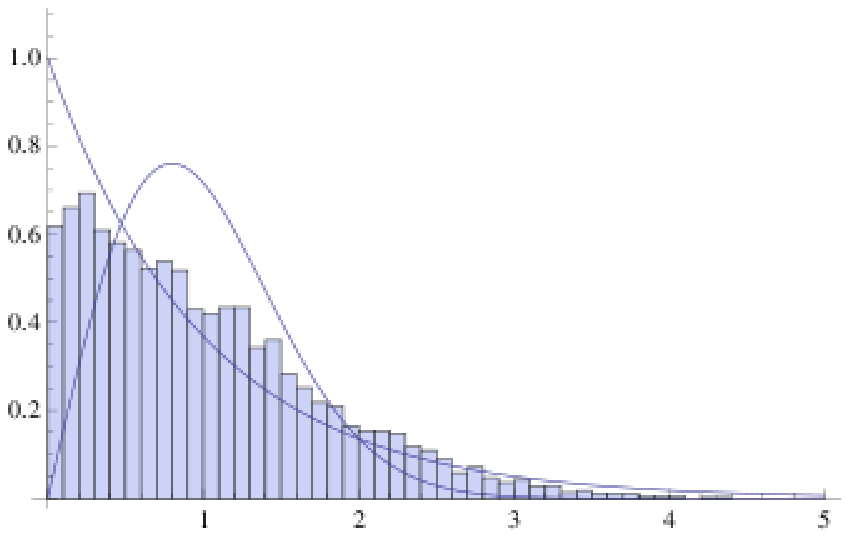}
\includegraphics[scale=.7]{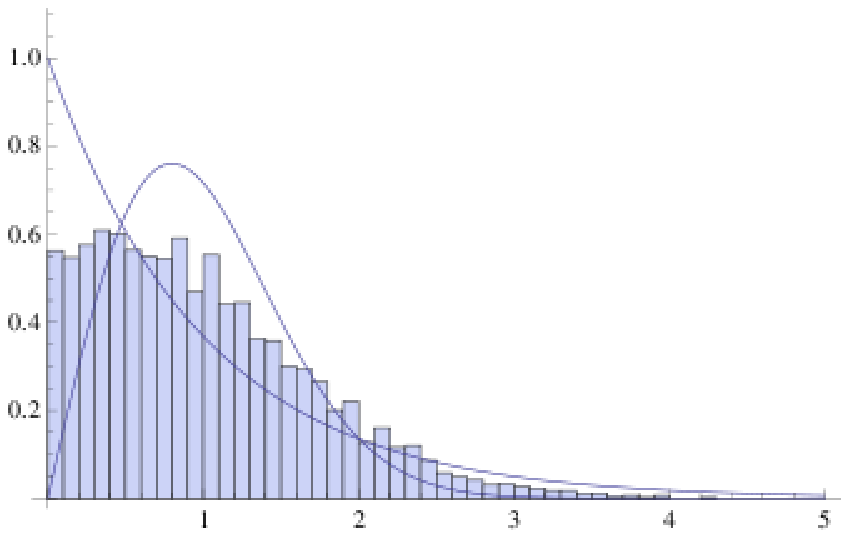}
\includegraphics[scale=.7]{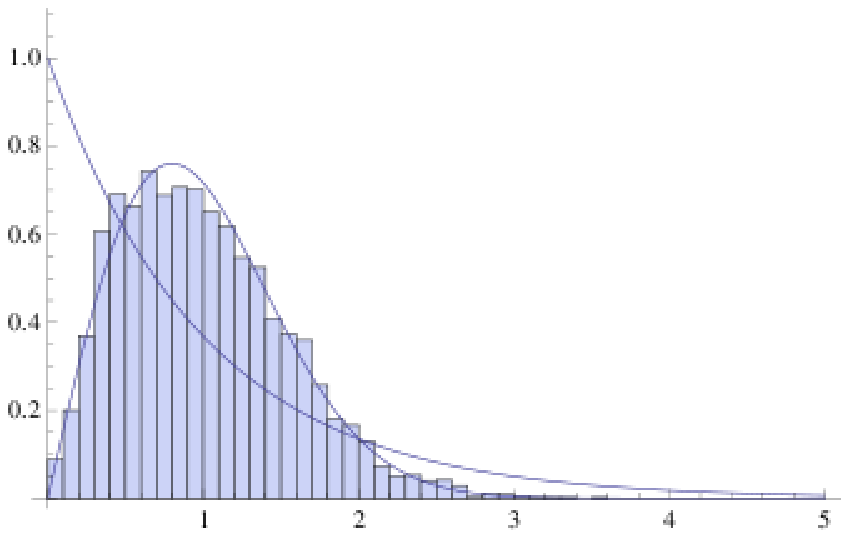}
\caption{\label{fig:normevsymmpermcirc} Density of nonzero spacings of the 10 central eigenvalues of 100 $1024 \times 1024$ symmetric $m$-block circulant matrices, with independent entries picked i.i.d.r.v. from a Gaussian, normalized to have mean spacing 1, with $m = 2,16,128,256,512,1024$, respectively. Compared to exponential and GOE densities.  %\textbf{IS THIS ALL MATRICES? HAVE YOU NORMALIZED TO HAVE MEAN SPACING 1? DO YOU WANT TO COMPARE SOME TO EXPONENTIAL AND SOME TO GOE, OR ACTUALLY ALL TO BOTH?}
}
\end{center}
\end{figure}

However, for symmetric $m$-block Toeplitz matrices, we see different behavior (see Figure \ref{fig:normevsymmpermcirctoep}).  The spacings look exponentially distributed for $m=1$ and appear to converge to the GOE distribution as we increase $m$.  In the Toeplitz case, but not in the circulant, we see the spacings behaving as the spectral densities do.

\begin{figure}
\begin{center}
\includegraphics[scale=.7]{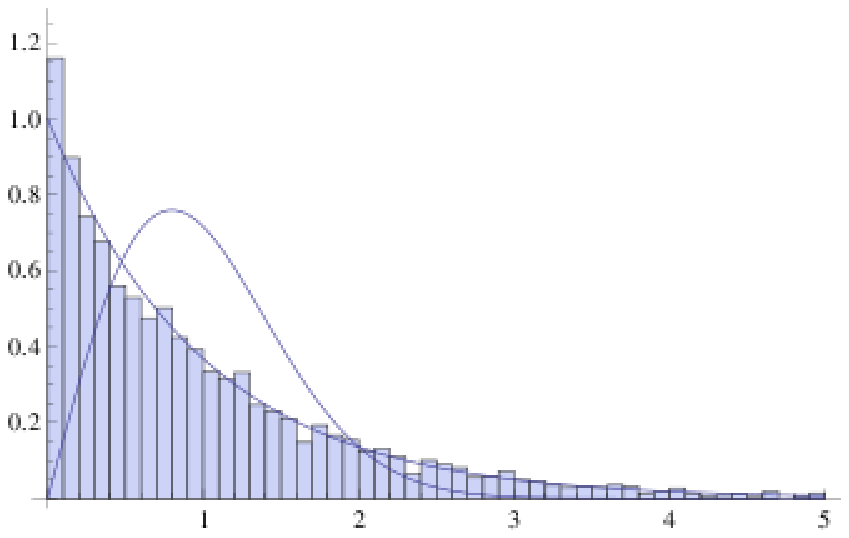}
\includegraphics[scale=.7]{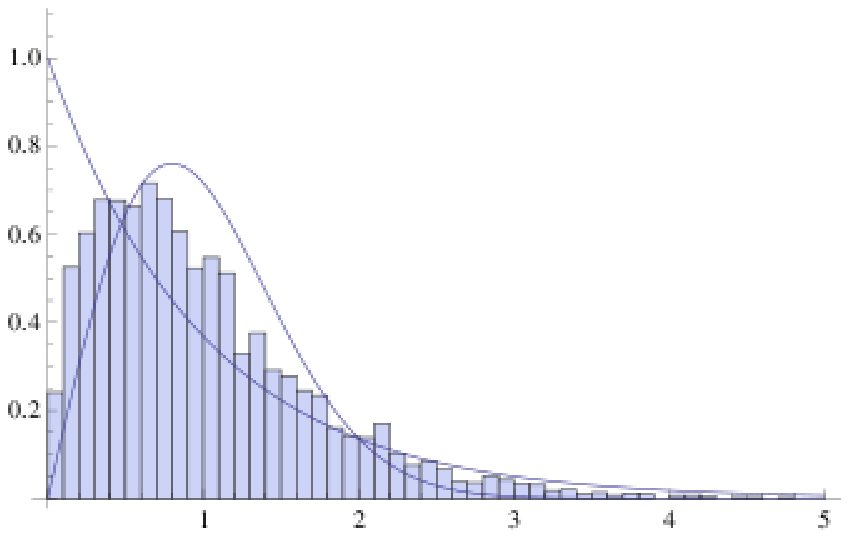}
\includegraphics[scale=.7]{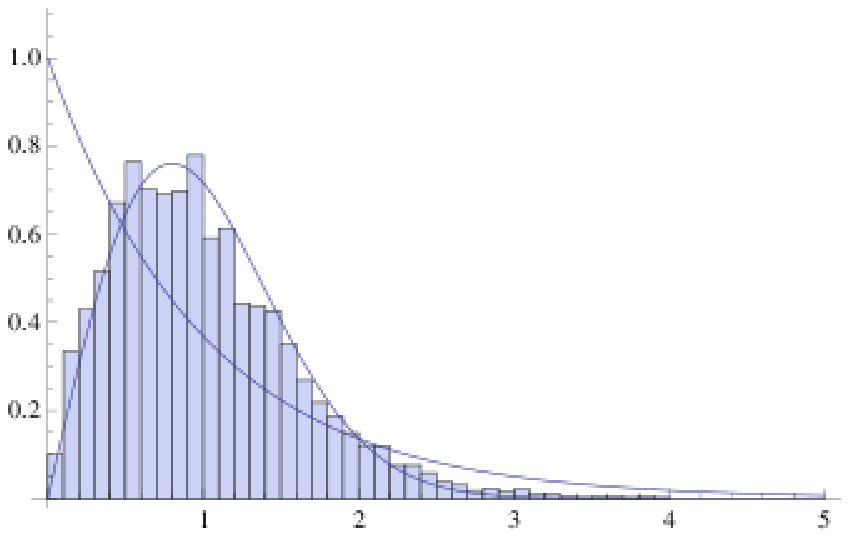}
\includegraphics[scale=.7]{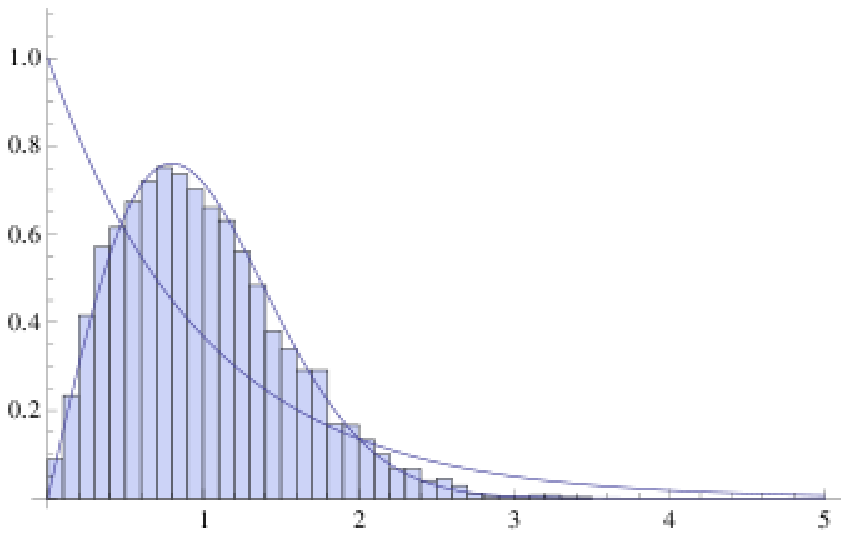}
\includegraphics[scale=.7]{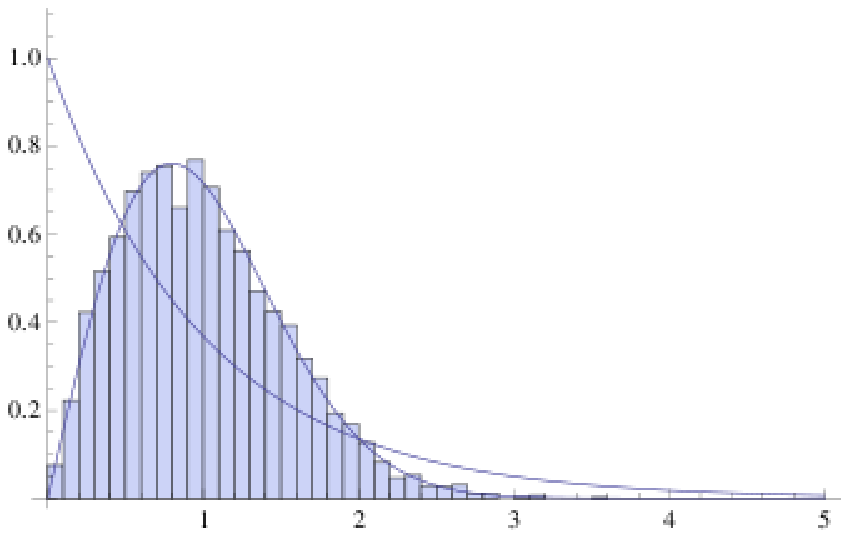}
\includegraphics[scale=.7]{diff1024c}
\caption{\label{fig:normevsymmpermcirctoep} Density of spacings of the 10 central eigenvalues of 100 $1024 \times 1024$ symmetric $m$-block Toeplitz matrices, with independent entries picked i.i.d.r.v. from a Gaussian, normalized to have mean spacing 1, with $m = 1,2,4,16,128,1024$, respectively. Compared to exponential and GOE densities.  %\textbf{IS THIS ALL MATRICES? HAVE YOU NORMALIZED TO HAVE MEAN SPACING 1? DO YOU WANT TO COMPARE SOME TO EXPONENTIAL AND SOME TO GOE, OR ACTUALLY ALL TO BOTH?}
}
\end{center}
\end{figure}

The representation theoretic approach will be used to solve the spacings problem for symmetric $m$-block circulant matrices in an upcoming paper of Kopp.  The spacing problem for block Toeplitz matrices will require some new innovation.

%%%%%%%%%%%%%%%%%%%%%%%%%%%%%%%%%%%%%%%%%%%%%%%%%%%%%%%%%%%%%%%%%%%%%%%%%%%%%%%%%%%%%%%%%%%%%%%%%%%%%%%%%%%%
%%%%%%%%%%%%%%%%%%%%%%%%%%%%%%%%%%%%%%%%%%%%%%%%%%%%%%%%%%%%%%%%%%%%%%%%%%%%%%%%%%%%%%%%%%%%%%%%%%%%%%%%%%%%

\appendix

\section{Pointwise Convergence as $m\to\infty$}\label{sec:convrate}

\emph{This appendix by Gene Kopp, Steven J. Miller and Frederick Strauch\footnote{Department of Physics, Williams College, fws1@williams.edu}.}\\

The characteristic function for the spectral measures of the period $m$-block circulant matrices is
\begin{equation}
\phi_m(t)\ =\ \frac{1}{m}\ e^{-\nicefrac{t^2}{2m}} \sum_{\ell=1}^m {m \choose \ell} \frac{1}{(\ell -1)!} \left( \frac{-t^2}{m}\right)^{\ell - 1},
\label{phi}
\end{equation}
which solves the differential equation \be t\phi_m''(t) + 3\phi_m'(t) + t\left(4 - \left(\frac{t}{m}\right)^2\right)\phi_m(t)\ =\ 0 \ee
with initial condition $\phi_m(0) = 1$; letting $m \to \infty$ gives $t\phi''(t) + 3\phi'(t) + 4t\phi(t) = 0$,
with initial condition $\phi(0) = 1$. The solution to the finite $m$ equation is a Laguerre polynomial, and the $m=\infty$ limit is $\nicefrac{J_1(2t)}{t}$ with $J_1$ the Bessel function of order 1.

To see this, recall that the generalized Laguerre polynomial (see \cite{AS}) has the explicit representation
\begin{equation}
L_n^{(\alpha)}(x)\ =\ \sum_{i=0}^n {n+ \alpha \choose n-i} \frac{1}{i!} (-x)^i.
\label{laguerre}
\end{equation}
To compare (\ref{phi}) with (\ref{laguerre}), we first shift the summation index by one ($\ell \mapsto \ell + 1$) to find
\begin{equation}
\phi_m(t)\ =\ \frac{1}{m} \ e^{-\nicefrac{t^2}{2m}} \sum_{\ell=0}^{m-1} {m \choose \ell+1} \frac{1}{\ell!} \left( \frac{-t^2}{m}\right)^{\ell}.
\end{equation}
Using the identity
\begin{equation}
{m \choose \ell + 1} = {m \choose m - 1 - \ell}
\end{equation}
we see that $n = m - 1$, $\alpha = 1$, and thus the characteristic function can be written in terms of the Laguerre polynomial:
\begin{equation}
\phi_{m}(t)\ =\ \frac{1}{m} e^{-\nicefrac{t^2}{2m}} L_{m-1}^{(1)} (\nicefrac{t^2}{m}),
\end{equation}
or equivalently in terms of the confluent hypergeometric function
\begin{equation}
\phi_{m}(t)\ =\ e^{-\nicefrac{t^2}{2m}} M(m+1,2,-\nicefrac{t^2}{m}).
\end{equation}

From 13.2.2 of \cite{AS} we have
$\lim_{m \to \infty} \phi_{m}(t) = \phi(t)$; however, we need some control on the rate of convergence.

%The following bounds for the difference between $\phi_m(t)$ and $\phi(t)$ are chosen for ease of exposition and are not the best possible, but more than suffice for our purposes (specifically, to prove Theorem \ref{thm:mainexpformulasconvergence}(3)); a more careful analysis, using 8.451(1) of \cite{GR} and analyzing separately the region $\frac13(\log\log m)^{1/2} \le t \le m^\beta$ for some $\beta < 1/2$ should allow us to replace $\log^{-1/4} m$ in the following lemma with $m^{-\delta}$ for some $\delta$.

\begin{lem}\label{lem:l1boundsphimphi} Let $r > \nicefrac{1}{3}$ and $\beta = \frac23(1-r)$. For all $m$ and all $t$ we have \be \twocase{\left|\phi_m(t) - \phi(t)\right| \ \ll_r \ }{m^{-(1-r)}}{if $|t| \le m^\beta$}{t^{-\nicefrac{3}{2}} + m^{-\nicefrac{5}{4}} \exp(-\nicefrac{t^2}{2m})}{otherwise,} \ee where the implied constant is independent of $m$ but may depend on $r$. This implies \be \int_{-\infty}^\infty \left|\phi_m(t) - \phi(t)\right|dt \ \ll \   m^{-\frac{1-r}{3}}. \ee Letting $\epsilon > 0$ and taking $r = \frac13 + 3\epsilon$ implies the integral is $O(m^{-\nicefrac{2}{9}+\epsilon})$.
\end{lem}

\begin{proof} We first consider small $t$: $|t| \le m^\beta$ with $\beta = \frac23(1-r)$. Using 13.3.7 of \cite{AS} with $a=m+1$, $b=2$ and $z=-\nicefrac{t^2}{m}$ to bound the confluent hypergeometric function $M$, we find \be\label{eq:ASlaguerretobessel} \phi_m(t) \ = \ e^{-\nicefrac{t^2}{2m}} M(m+1,2,-\nicefrac{t^2}{m}) \ = \ \frac{J_1(2t)}{t} + \sum_{n=1}^\infty A_n (2m)^{-n} (-1)^n t^{n-1} J_{n+1}(2t),\ee where $A_0 = 1$, $A_1 = 0$, $A_2 = 1$ and $A_{n+1} = A_{n-1}+ \frac{2m}{n+1} A_{n-2}$ for $n \ge 2$.

For any $r>\nicefrac{1}{3}$ and $m$ sufficiently large we have $A_n \le m^{rn}$ (we can't do better than $r>\nicefrac{1}{3}$ as $A_3=\frac23 m$). This follows by induction. It is clear for $n\le 2$, and for larger $n$ we have by the inductive assumption that \be A_{n+1} \ = \ A_{n-1} +  \frac{2m}{n+1} A_{n-2} \ \le \ m^{r(n-1)} + m \cdot m^{r(n-2)} \ = \ m^{r(n+1)} \cdot (m^{-2r} + m^{1-3r}); \ee as $r > \nicefrac{1}{3}$ the above is less than $m^{r(n+1)}$ for $m$ large. If we desire a bound to hold for all $m$, we instead use $A_n \le c_r m^{rn}$ for $c_r$ sufficiently large. Substituting this bound for $A_n$ into \eqref{eq:ASlaguerretobessel}, noting $\nicefrac{J_1(2t)}{t} = \phi(t)$ and using $|J_n(x)| \le 1$ (see 9.1.60 of \cite{AS}) yields, for $|t| \le m^{1-r}$, \bea \left|\phi_m(t) - \phi(t)\right| & \ \le \ & \frac{c_r}{2m^{1-r}} \sum_{n=1}^\infty \left(\frac{t}{2m^{1-r}}\right)^{n-1} \ \ll_r \ m^{-(1-r)}.\eea

We now turn to $t$ large: $|t| \ge  m^\beta$. Using \be \left|\phi_m(t) - \phi(t)\right| \ \le \ \left|\phi_m(t)\right| + \left|\phi(t)\right| \ee to trivially bound the difference, the claim follows the decay of the Bessel and Laguerre functions. Specifically, (see 8.451(1) of \cite{GR}) we have $J_1(x) \ll x^{-\nicefrac{1}{2}}$ and thus \be \phi(t) \ = \ \frac{J_1(2t)}{t} \ \ll \ t^{-\nicefrac{3}{2}}.\ee For $\phi_m(t)$, we use 8.978(3) of \cite{GR}, which states \be L_n^{(\alpha)}(x) \ = \ \pi^{-\nicefrac{1}{2}} e^{\nicefrac{x}{2}} x^{-\nicefrac{\alpha}{2} - \nicefrac{1}{4}} n^{\nicefrac{\alpha}{2}-\nicefrac{1}{4}} \cos\left(2\sqrt{nx}-\frac{\alpha \pi}{2}-\frac{\pi}{4}\right) + O\left(n^{\nicefrac{\alpha}{2}-\nicefrac{3}{4}}\right), \ee so long as ${\rm Im}(\alpha) = 0$ and $x>0$. Letting $x = \nicefrac{t^2}{m}$ with $|t| \ge \frac13\log^{\nicefrac{1}{2}} m$, $\alpha = 1$ and $n=m-1$ we find \bea \phi_m(t) & \ = \ & m^{-1} e^{-\nicefrac{t^2}{2m}} L_{m-1}^{(1)}(\nicefrac{t^2}{m})\nonumber\\ & \ \ll \ & m^{-1} e^{-\nicefrac{t^2}{2m}} \left[e^{\nicefrac{t^2}{2m}} (\nicefrac{t^2}{m})^{-\nicefrac{3}{4}} m^{\nicefrac{1}{4}} + m^{-\nicefrac{1}{4}}\right] \nonumber\\ & \ll & t^{-\nicefrac{3}{2}} + m^{-\nicefrac{5}{4}} e^{-\nicefrac{t^2}{2m}}. \eea

All that remains is to prove the claimed bound for $\int_{-\infty}^\infty \left|\phi_m(t) - \phi(t)\right|dt$. The contribution from $|t| \le m^\beta$ is easily seen to be $O_r(\nicefrac{m^\beta}{m^{1-r}}) =O_r(m^{-\nicefrac{(1-r)}{3}})$ with our choice of $\beta$. For $|t| \ge m^\beta$, we have a contribution bounded by \bea 2\int_{m^\beta}^\infty \left(t^{-\nicefrac{3}{2}} + m^{-\nicefrac{5}{4}} e^{-\nicefrac{t^2}{2m}}\right)dt & \ \ll  \ & m^{-\nicefrac{\beta}{2}} + m^{-\nicefrac{3}{4}} \int_{-\infty}^\infty \frac1{\sqrt{2\pi m}}\ \exp(-\nicefrac{t^2}{2m})dt \nonumber\\ & \ll & m^{-\nicefrac{(1-r)}{3}} + m^{-\nicefrac{3}{4}}, \eea as the last integral is that of a Gaussian with mean zero and variance $m$ and hence is 1. (We chose $\beta = \frac23(1-r)$ to equalize the bounds for the two integrals.)  \end{proof}

%we use (13.5.1) and (13.5.4) of \cite{AS} (we do not need to use (13.5.3) as the factor $\Gamma(b-a)$ in the denominator of (13.5.1) has a pole as $b-a=1-m$ is a negative integer, and thus this term vanishes). We need (13.5.4) as it gives the $a$ and $b$ dependence of the big-Oh term in (13.5.1), which for $M(m+1,2,-t^2/m)$ is trivially $$O\left(m^2 \cdot \frac{m}{t^2} \cdot \left(m + \frac{t^2}{m} + \frac{m}{t^2}\right)\right) \ = \ O\left(m^4\right)$$ as $|t| \ge 1$. Thus  \bea \phi_m(t) & \ = \ & e^{t^2/2m} M(m+1,2,-t^2/m)\nonumber\\ & \ = \ & e^{t^2/2m} e^{-t^2/m} \frac{\Gamma(2)}{\Gamma(m+1)} (-t^2/m)^{m-1} \left[1 + O(-t^2/m)\right].\eea

%%%%%%%%%%%%%%%%%%%%%%%%%%%%%%%%%%%%%%%%%%%%%%%%%%%%%%%%%%%%%%%%%%%%%%%%%%%%%%%%%%%%%%%%%%%%%%%%%%%%%%%%%%%%
%%%%%%%%%%%%%%%%%%%%%%%%%%%%%%%%%%%%%%%%%%%%%%%%%%%%%%%%%%%%%%%%%%%%%%%%%%%%%%%%%%%%%%%%%%%%%%%%%%%%%%%%%%%%

\section{Generalized $m$-Block Circulant Matrices}\label{sec:appwentao}

\emph{This appendix by Steven J. Miller and Wentao Xiong\footnote{Department of Mathematics and Statistics, Williams College, xx1@williams.edu}.}\\

As the proofs are similar to the proof for $m$-block circulant matrices, we just highlight the differences. The trace expansion from before holds, as do the arguments that the odd moments vanish.

We first explore the modulo condition to compute some low moments, and show that the difference in the modulo condition between the $m$-block circulant matrices and the generalized $m$-block circulant matrices leads to different values for moments, and hence limiting spectral distributions. Thus the limiting spectral distribution depends on the frequency of each element, as well as the way the elements are arranged, in an $m$-pattern.

\subsection{Zone-wise Locations and Pairing Conditions}

Since we have restricted the computation of moments to even moments, and have shown that the only configurations that contribute to the $2k$\textsuperscript{th} moment are those in which the $2k$ matrix entries are matched in $k$ pairs in opposite orientation, we are ready to compute the moments explicitly. We start by calculating the $2$\textsuperscript{nd} moment, which by \eqref{moment} is $\frac{1}{N^2}\sum_{1\le i,j\le N} a_{ij}a_{ji}$. As long as the matrix is symmetric, $a_{ij}=a_{ji}$ and the $2$\textsuperscript{nd} moment is $1$. We now describe the conditions for two entries $a_{i_{s}i_{s+1}}, a_{i_{t}i_{t+1}}$ to be paired, denoted as  $a_{i_{s}i_{s+1}}=a_{i_{t}i_{t+1}} \equi (s,s+1)\sim (t,t+1)$, which we need to consider in detail for the computation of higher moments. To facilitate the practice of checking pairing conditions, we divide an $N\times N$ symmetric $m$-block circulant matrix into $4$ zones (see Figure \ref{fig:zones}), and then reduce an entry $a_{i_{s}i_{s+1}}$ in the matrix to its ``basic form''. Write $i_\ell =m\eta_\ell +\epsilon_\ell$, where $\eta_\ell \in\{1,2,\dots ,\frac{N}{m}\}$ and $\epsilon_\ell \in\{0,1,\dots ,m-1\}$, we have
%\begin{figure}
%\begin{center}
%\label{fig:4zones}
%\includegraphics[scale=0.64]{4zones.jpg}
%\caption{$4$ zones in an $N\times N$ matrix.}
%\end{center}
%\end{figure}
\ben\label{zones}
\item $0\leq i_{s+1}-i_{s}\leq\frac{N}{2}-1 \Rightarrow a_{i_{s}i_{s+1}}\in$ zone 1 and $a_{i_{s}i_{s+1}} =a_{\epsilon_s ,m(\eta_{s+1}-\eta_s)+\epsilon_{s+1}}$;

\item $\frac{N}{2}\leq i_{s+1}-i_{s}\leq N-1 \Rightarrow a_{i_{s}i_{s+1}}\in$ zone 2 and $a_{i_{s}i_{s+1}} =a_{\epsilon_{s+1} ,m(\eta_s+\frac{N}{m}-\eta_{s+1})+\epsilon_s}$;

\item $\frac{N}{2}\leq i_{s}-i_{s+1}\leq N-1 \Rightarrow a_{i_{s}i_{s+1}}\in$ zone 3 and $a_{i_{s}i_{s+1}} =a_{\epsilon_s ,m(\eta_{s+1}+\frac{N}{m}-\eta_s)+\epsilon_{s+1}}$;

\item $0\leq i_{s}-i_{s+1}\leq\frac{N}{2}-1 \Rightarrow a_{i_{s}i_{s+1}}\in$ zone 4 and $a_{i_{s}i_{s+1}} =a_{\epsilon_{s+1} ,m(\eta_s-\eta_{s+1})+\epsilon_s}$.
\een
In short, $(i_{s+1}-i_s)$ determines which diagonal $a_{i_{s}i_{s+1}}$ is on. If $a_{i_{s}i_{s+1}}$ is in zone 1 or 3 (Area I), $\epsilon_s$ determines the slot of $a_{i_{s}i_{s+1}}$ in an $m$-pattern; if $a_{i_{s}i_{s+1}}$ is in zone 2 or 4 (Area II), $\epsilon_{s+1}$ determines the slot of $a_{i_{s}i_{s+1}}$ in an $m$-pattern.

\begin{center}
\begin{figure}
\centering\includegraphics[scale=.42]{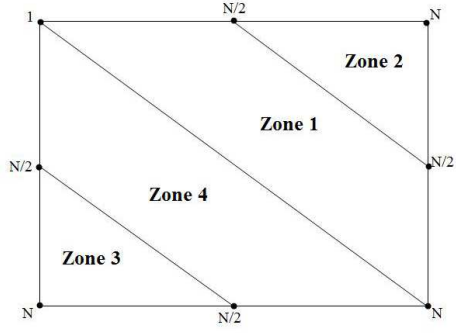}
\caption{The four zones for $m$-block circulant matrices.}\label{fig:zones}
\end{figure}
\end{center}

Recall the two basic pairing conditions, the diagonal condition that we have explored before, and the modulo condition, for which we will define an equivalence relation $\mathcal{R}$. For a real symmetric $m$-block circulant matrix following a generalized $m$-pattern and any two entries $a_{i_{s}i_{s+1}},a_{i_{t}i_{t+1}}$ in the matrix, suppose that $i_{s}$ and $i_{t+1}$ are the indices that determine the slot of the respective entries, then $i_{s}\mathcal{R}i_{t+1}$ if and only if $a_{i_{s}i_{s+1}},a_{i_{t}i_{t+1}}$ are in certain slots in an $m$-pattern such that these two entries can be equal. For example, for the $\{a,b\}$ pattern, $i_{s}\mathcal{R}i_{t+1}\equi i_{s}\con i_{t+1}\Mod{2}$; for the $\{a,a,b,b\}$ pattern, $i_{s}\mathcal{R}i_{t+1}\equi\mod{(i_{s},4)},\mod{(i_{t+1},4)}\in\{1,2\} \mbox{ or }\mod{(i_{s},4)},\mod{(i_{t+1},4)}\in\{3,0\}$.

We now formally define the two pairing conditions.
\ben\label{pairconds}
\item (diagonal condition) $i_{s} -i_{s+1} \equiv -(i_{t} -i_{t+1}) \pmod{N}$.

\item (modulo condition) $i_s\mathcal{R}i_{t+1}$ or $i_{s+1}\mathcal{R}i_t$, depending on which zone(s) $a_{i_{s}i_{s+1}}, a_{i_{t}i_{t+1}}$ are located in.
\een
Since the diagonal condition implies a Diophantine equation for each of the $k$ pairs of matrix entries, we only need to choose $k+1$ out of $2k$ $i_\ell$'s, and the remaining $i_\ell$'s are determined. This shows that, trivially, the number of non-trivial configurations is bounded above by $N^{k+1}$.
In addition, the diagonal condition always ensure that $a_{i_{s}i_{s+1}}$ and $a_{i_{t}i_{t+1}}$ are located in different areas. For instance, if $a_{i_{s}i_{s+1}}\in$ zone 1 and $i_{s}-i_{s+1}=-(i_{t}-i_{t+1})$, then $a_{i_{s}i_{s+1}}\in$ zone 4; if $a_{i_{s}i_{s+1}}\in$ zone 1 and $i_{s}-i_{s+1} =-(i_{t}-i_{t+1})-N$, then $a_{i_{s}i_{s+1}}\in$ zone 2, etc. Thus, if $i_s$ determines the slot for $a_{i_{s}i_{s+1}}$ in an $m$ pattern, then $i_{t+1}$ determines for $a_{i_{t}i_{t+1}}$; if $i_{s+1}$ determines the slot for $a_{i_{s}i_{s+1}}$, then $i_{t}$ determines for $a_{i_{t}i_{t+1}}$, and vice versa.

Considering the ``basic'' form of the entries, the two conditions above are equivalent to
\ben\label{circpairconds}
\item (diagonal condition) $(m\eta_{s}+\epsilon_{s})-(m\eta_{s+1}+\epsilon_{s+1}) \equiv -(m\eta_{t}+\epsilon_{t})+(m\eta_{t+1}+\epsilon_{t+1})\Mod{N}$
$\Rightarrow m(\eta_{s}-\eta_{s+1}+\eta_{t}-\eta_{t+1})+(\epsilon_{s}-\epsilon_{s+1}+\epsilon_{t}-\epsilon_{t+1}) =0\text{ or }\pm N$.

\item (modulo condition) $\epsilon_{s}\mathcal{R}\epsilon_{t+1}$ or $\epsilon_{s+1}\mathcal{R}\epsilon_{t}$.
\een
Since $m|N$, this requires $m|(\epsilon_{s}-\epsilon_{s+1}+\epsilon_{t}-\epsilon_{t+1})$. Given the range of the $\eta_\ell$'s and $\epsilon_\ell$'s, we have $\epsilon_{s}-\epsilon_{s+1}+\epsilon_{t}-\epsilon_{t+1} =0\text{ or }\pm m$, which indicates that
\be \eta_{s}-\eta_{s+1}+\eta_{t}-\eta_{t+1} =0, \pm 1, \frac{N}{m}, \frac{N}{m}\pm 1, -\frac{N}{m},\text{ or }-\frac{N}{m}\pm 1.\ee

As discussed before, if we allow repeated elements in an $m$-pattern, the equivalence relation $\mathcal{R}$ no longer necessitates a congruence relation as in pattern where each element is distinct. While the computation of high moments for general $m$-patterns appears intractable, fortunately we are able to illustrate how the difference in the modulo condition affects moment values by comparing the low moments for two simple patterns $\{a,b,a,b\}$ and $\{a,a,b,b\}$.

%%%%%%%%%%%%%%%%%%%%%%%%%%%%%%%%
%%%%%%%%%%%%%%%%%%%%%%%%%%%%%%%%

\subsection{The Fourth Moment}

Although we will show that the higher moments differ by the way the elements are arranged in an $m$-pattern, the $4$\textsuperscript{th} moment is in fact independent of the arrangement of elements. We first show that the $4$\textsuperscript{th} moment for any $m$-pattern is determined solely by the frequency at which each element appears, and then show that this lemma fails for the $6$\textsuperscript{th} moment and higher (see also \cite{Xi}). Briefly, for the higher moments for patterns with repeated elements, there exist ``obstructions to modulo equations'' that make trivial some non-trivial configurations for patterns without repeated elements. Due to the obstructions to modulo equations, some configurations that are non-trivial for all-distinct patterns become trivial for patterns with repeated elements, making the higher moments for repeated patterns smaller.

\begin{lem}\label{4thmom}
For an ensemble of real symmetric period $m$-block circulant matrices of size $N$, if within each $m$-pattern we have $n$ i.i.d.r.v. $\{\alpha_r\}_{r=1}^n$, each of which has a fixed number of occurrences $\nu_r$ such that $\sum_{r=1}^n \nu_r =m$, the $4$\textsuperscript{th} moment of the limiting spectral distribution is $2+\sum_{r=1}^n (\frac{\nu_r}{m})^3$.
\end{lem}

By \eqref{moment}, we calculate $\frac{1}{N^{\frac{4}{2}+1}}\sum_{1\leq i,j,k,l\leq N} a_{ij}a_{jk}a_{kl}a_{li}$ for the $4$\textsuperscript{th} moment. There are $2$ ways of matching the $4$ entries in $2$ pairs:

\ben\label{4thmommat}
\item (adjacent, 2 variations) $a_{ij}=a_{jk}$ and $a_{kl}=a_{li}$ (or equivalently $a_{ij} = a_{li}$ and $a_{jk} = a_{kl}$);

\item (diagonal, 1 variation) $a_{ij}=a_{kl}$ and $a_{jk}=a_{li}$.
\een

there are $3$ matchings, with the two adjacent matchings contributing the same to the $4$\textsuperscript{th} moment. We first consider one of the adjacent matchings, $a_{ij}=a_{jk}$ and $a_{kl}=a_{li}$. The pairing conditions \eqref{pairconds} in this case are:
\ben
\item (diagonal condition) $i-j\con k-j \Mod{N}$, $k-l\con i-l\Mod{N}$;

\item (modulo condition) $i\mathcal{R}k \mbox{ or }j\mathcal{R}j$, $k\mathcal{R}i \mbox{ or }l\mathcal{R}l$.
\een
Since $1\leq i,j,k,l\leq N$, the diagonal condition requires $i=k$, and then the modulo condition follows trivially, regardless of the $m$-pattern we study. Hence, we can choose $j$ and $l$ freely, each with $N$ choices, $i$ freely with $N$ choices, and then $k$ is fixed. This matching then contributes $\frac{N^3}{N^{\frac{4}{2}+1}}=1$ (fully) to the $4$\textsuperscript{th} moment, so does the other adjacent matching.

We proceed to the diagonal matching, $a_{ij}=a_{kl}$ and $a_{jk}=a_{li}$. The pairing conditions \eqref{pairconds} in this case are:
\ben
\item (diagonal condition) $i-j\con l-k \Mod{N}$, $j-k\con i-l \Mod{N}$;

\item (modulo condition) $i\mathcal{R}l \mbox{ or }j\mathcal{R}k$, $j\mathcal{R}i \mbox{ or }k\mathcal{R}l$.
\een
The diagonal condition $j-k\con i-l \Mod{N}$ is equivalent to $i-j\con l-k \Mod{N}$, which entails
\ben\label{4thmomdiag}
\item $i+k =j+l$, or

\item $i+k =j+l+N$, or

\item $i+k =j+l-N$.
\een
In any case, we only need to choose 3 indices out of $i,j,l,k$, and then the last one is fixed. In the following argument, without loss of generality, we choose $(i,j,l)$ and thus fix $k$.

For a general $m$-pattern, we write $i =4\eta_1 +\gep_1 ,j =4\eta_2 +\gep_2 ,k =4\eta_3 +\gep_3 ,l =4\eta_4 +\gep_4$, where $\eta_1 ,\eta_2 ,\eta_3 ,\eta_4 \in\{0,1,\dots,\frac{N}{m}\}$ and $\gep_1 ,\gep_2 ,\gep_3 ,\gep_4\in \{0,1,\dots, m-1\}$. Before we consider the $\epsilon_\ell$'s, we note that there exist Diophantine constraints. For example, if $i+k =j+l$, given that $1\leq i,j,l\leq N$, $k =j+l-i$ also needs to satisfy $1\leq k\leq N$. As a result, we need $0\leq \eta_2 +\eta_4 -\eta_1 \leq\frac{N}{4}$. Note that, due to the $\gep_\ell$'s, sometimes we may have $0\leq \eta_2 +\eta_4 -\eta_1 \leq\frac{N}{4}+\varepsilon$, where the error term $\varepsilon\in (-\frac{m}{2},\frac{m}{2})$ and only trivially affects the number of choices of $(\eta_2,\eta_4,\eta_1)$ for a fixed $m$ as $N\to\infty$.

We now explore the Diophantine constraints for each variation of the diagonal condition \eqref{4thmomdiag}. The $i+k=j+l$ case is similar to that in \cite{HM}, where, in a Toeplitz matrix, the diagonal condition only entails $i+k=j+l$, and there are obstructions to the system of Diophantine equations following the diagonal condition. However, the circulant structure that adds $i+k =j+l+N$ and $i+k =j+l-N$ to the diagonal condition fully makes up the Diophantine obstructions. This explains why the limiting spectral distribution for ensembles of circulant matrices has the moments of a Gaussian, while that for ensembles of Toeplitz matrices has smaller even moments. We now study the $3$ possibilities of the diagonal condition for the circulant structure.
\ben
\item Consider $i+k=j+l$.  We use Lemma 2.5 from \cite{HM} to handle the obstructions to Diophantine equations, which says:
\emph{Let $I_N = \{1,\dots, N\}$. Then $\#\{x,y,z \in I_N: 1\leq x+y-z \leq N\} =\frac{2}{3}N^3 + \frac{1}{3}N$.}

In our case, let $M=\frac{N}{m}$. The number of possible combinations of $(\eta_2,\eta_4,\eta_1)$ that allow $0\leq\eta_3\leq\frac{N}{4}$ is $\frac{2}{3}M^3 +\frac{1}{3}M$.\footnote{ In \cite{HM}, the related lemma is proven for $\eta_2,\eta_4,\eta_1 \in\N_+$, i.e., no cases where $\eta_2\eta_4\eta_1 =0$. Thus we are supposed to start from $S=0$; however, as $N\to\infty$, the error from this becomes negligible. }  For each of $\eta_2,\eta_4,\eta_1$, we have $m$ free choices of $\gep_\ell$, and thus the number of $(i,j,l)$ is $m^3 (\frac{2}{3}M^3 +\frac{1}{3}M) =\frac{2}{3}N^3 +O(N)$.

\item Consider $i+k =j+l+N$. Note $1\leq k\leq N$ requires $0\leq \eta_2 +\eta_4 -\eta_1 +\frac{N}{m}\leq\frac{N}{m} \Rightarrow -\frac{N}{m}\leq \eta_2 +\eta_4 -\eta_1\leq 0$. Similar to the $i+k=j+l$ case, we write $M =\frac{N}{m}$ and $S =\eta_2 +\eta_4$, and then $-\frac{N}{m}\leq S -\eta_1\leq 0 \Rightarrow S\leq\eta_1\leq M+S$ where obviously $S\leq M$. We have $S+1$ ways to choose $(\eta_2,\eta_4)$ s.t. $\eta_2+\eta_4 =S$, and $M-S+1$ choices of $\eta_1$. The number of $(i,j,l)$ is thus
\be m^3\sum_{S=0}^M (S+1)(M-S+1) =m^3 \left(\frac{M^3}{6} +M^2 +\frac{5}{6}M\right) =\frac{N^3}{6} +O(N^2).\ee

\item Consider $i+k =j+l-N$. Now $1\leq k\leq N$ requires $0\leq \eta_2 +\eta_4 -\eta_1 -\frac{N}{m}\leq\frac{N}{m} \Rightarrow \frac{N}{m}\leq \eta_2 +\eta_4 -\eta_1 \leq \frac{2N}{m}$. Again, we write $M =\frac{N}{m}$ and $S =\eta_1 +\eta_4$, and then $M\leq S -\eta_1\leq 2M \Rightarrow S-2M\leq\eta_1\leq S-M$ where obviously $S\geq M$. We have $2M-S+1$ ways to choose $(\eta_2,\eta_4)$ s.t. $\eta_2+\eta_4 =S$, and $S-M+1$ choices of $\eta_1$. The number of $(i,j,l)$ is thus
\be m^3 \sum _{S=M}^{2M} (2M-S+1)(S-M+1) =m^3\left(\frac{M^3}{6} +M^2 +\frac{5}{6}M\right) =\frac{N^3}{6} +O(N^2).\ee
\een
Therefore, with the additional diagonal conditions $i+k =j+l+N$ and $i+k =j+l-N$ induced by the circulant structure, the number of $(i,j,l)$ is of the order $(\frac{2}{3}+\frac{1}{6}+\frac{1}{6})N^3=N^3$, i.e. the circulant structure makes up the obstructions to Diophantine equations in the Toeplitz case. Since the $\eta_\ell$'s do not matter for the modulo condition, to make a non-trivial configuration, we may choose three $\eta_\ell$'s freely, each with $\frac{N}{m}$ choices, and then choose some $\gep_\ell$'s that satisfy the modulo condition, which we will study below.

For the modulo condition, it is necessary to figure out which zones the four entries are located in. Recall that the diagonal condition will always ensure that two paired entries are located in different areas. For the $4$\textsuperscript{th} moment, each of the $3$ variations of the diagonal condition is sufficient to ensure that any pair of entries involved are located in the right zones. We may check this rigorously by enumerating all possibilities of the zone-wise locations of the $4$ entries, e.g. if $i+k =j+l+N$, if $a_{ij}\in$ zone 1, then $a_{kl}\in$ zone 2.\footnote{ This enumeration is complicated since the zone where an entry $a_{ij}$ is located imposes restrictions on the choice of $i,j$, e.g. when $a_{i,j}\in$ zone 2, we have $i\geq\frac{N}{2}$ and $j\leq\frac{N}{2}$. } As a result, for a pair of matrix elements in the diagonal matching, say $a_{ij}=a_{kl}$, if $i$ determines the slot in an $m$-pattern for $a_{ij}$ and thus matters for the modulo condition, then $l$ determines for $a_{kl}$; if $j$ determines for $a_{ij}$, then $k$ determines for $a_{kl}$, and vice versa.

With the zone-wise issues settled, we study how to obtain a non-trivial configuration for the $4$\textsuperscript{th} moment. Recall the modulo condition for the diagonal matching: $i\mathcal{R}l \mbox{ or }j\mathcal{R}k$, $j\mathcal{R}i \mbox{ or }k\mathcal{R}l$. This entails $2^2 =4$ sets of equivalence relations,
\be i\mathcal{R}l\mathcal{R}j; i\mathcal{R}l\mathcal{R}k, j\mathcal{R}k\mathcal{R}i, j\mathcal{R}k\mathcal{R}l \ee
Each set of equivalence relations appears with a certain probability, depending on the zone-wise locations of the $4$ entries. For example, $i\mathcal{R}l\mathcal{R}j$ follows from $i\mathcal{R}l$ and $j\mathcal{R}i$, which requires both $a_{ij}$ and $a_{jk} \in$ Area I. Regardless of the probability with which each set occurs, we choose one free index with $N$ choices, and then another two indices such that these $3$ indices are related to each other under $\mathcal{R}$. The number of choices of the two indices after the free one is determined solely by the number of occurrences of the elements in an $m$-pattern.

We give a specific example of making a non-trivial configuration for the $4$\textsuperscript{th} for two simple patterns $\{a,b,a,b\}$ and $\{a,a,b,b\}$. Under the condition $i+k =j+l$, if $a_{ij}\in$ zone 1 and $a_{jk}\in$ zone 3, then $a_{kl}\in$ zone 4 and $a_{li}\in$ zone 2. We first select $\eta_1 ,\eta_2, \eta_4$ such that $i,j,l$ and $k =j+l-i$ satisfy the zone-wise locations.\footnote{ It is noteworthy that the specific location of an element still depends on the $\gep_\ell$'s, but as $N\to\infty$, the probability that the $\eta_\ell$'s alone determine the zone-wise locations of elements approaches $1$, i.e. the probability that adding the $\gep_\ell$'s changes the zone-wise location of an element approaches $0$. } In this case, based on pairing conditions \eqref{pairconds}, pairing $a_{ij}=a_{kl}$ and $a_{jk}=a_{li}$ will require $\gep_1\mathcal{R}\gep_4$ and $\gep_2\mathcal{R}\gep_1$, or equivalently $\gep_1\mathcal{R}\gep_2\mathcal{R}\gep_4$. Without loss of generality, we can start with a free $\gep_1$ with $4$ choices, then there are $2$ free choices for each of $\gep_2$ and $\gep_4$, and then we have a non-trivial configuration. We have similar stories under the other two variations of the diagonal condition and with other zone-wise locations of $a_{ij}$ and $a_{kl}$. Therefore, we can choose three out of four $\eta_\ell$'s freely, each with $\frac{N}{4}$ choices, then one $\gep_\ell$ with $4$ choices, then another two $\gep_\ell$'s each with $2$ choices, and finally the last index is determined under the diagonal condition. As discussed before, such a choice of indices will always satisfy the zone-wise requirements and thus the $\gep$-based pairing conditions. Thus there are $(\frac{N}{4})^3 \cdot 4\cdot 2\cdot 2 =\frac{N^3}{4}$ choices of $(i,j,k,l)$ that will produce a non-trivial configuration. It follows that the contribution from the diagonal matching to the $4$\textsuperscript{th} moment is $\frac{1}{N^3}(\frac{2}{3} +\frac{1}{6}+\frac{1}{6})\frac{N^3}{4} =\frac{1}{4}$.

The computation of the $4$\textsuperscript{th} moment for the simple patterns $\{a,b,a,b\}$ and $\{a,a,b,b\}$ can be immediately generalized to the $4$\textsuperscript{th} moment for other patterns. As emphasized before, both adjacent matchings contribute fully to the $4$\textsuperscript{th} moment regardless of the $m$-pattern. For diagonal matching, the system of Diophantine equations induced by the diagonal condition are also independent of the $m$-pattern in question, and the way we count possible configurations can be easily generalized to an arbitary $m$-pattern. We have thus proved Lemma \ref{4thmom}.

Note that Lemma \ref{4thmom} implies that the $4$\textsuperscript{th} moment for any pattern depends solely on the frequency at which each element appears in an $m$-period. Besides the $\{a,a,b,b\}$ pattern that we have studied in depth, we may easily test two extreme cases. One case where $n=m$, i.e. each random variable appears only once, represents the $m$-block circulant matrices from Theorem \ref{thm:mainexpformulasconvergence} for which the $4$\textsuperscript{th} moment is $2+\frac{1}{m^2}$ (and $m=1$ represents the circulant matrices for which the $4$\textsuperscript{th} moment is $3$). Numerical simulations for numerous patterns including $\{a,a,b\}$, $\{a,b,b\}$, $\{a,b,b,a\}$, $\{a,b,c,a,b,c\}$, $\{a,b,c,d,e,e,d,c,b,a\}$ et cetera support Lemma \ref{4thmom} as well; we present results of some simulations in Tables \ref{table:abpattern1} to \ref{table:abpattern3}.

\begin{center}
\begin{table}
\begin{tabular}{|r||r||r|r|r||r|}
\hline
$k$	&	$abab$ (theory)	&	$abab$ (observed)	&	$aabb$ (observed)	&	$abba$ (observed)	&	$N(0,1)$	\\
\hline											
2	&	1.0000	&	1.0016	&	1.0014	&	0.9972	&	1	\\
4	&	2.2500	&	2.2583	&	2.2541	&	2.2405	&	3	\\
6	&	7.5000	&	7.5577	&	7.3212	&	7.2938	&	15	\\
8	&	32.8125	&	33.2506	&	30.4822	&	30.5631	&	105	\\
10	&	177.1880	&	180.8270	&	153.9530	&	155.6930	&	945	\\
\hline											
\end{tabular}
\caption{Comparison of moments for various patterns involving $a$ and $b$. The first column are the theoretical values for the moments of the pattern ${a,b}$, and the final are the moments of the standard normal. The middle three columns are 200 simulations of $4000 \times 4000$ matrices. }\label{table:abpattern1}
\end{table}
\end{center}

\begin{center}
\begin{table}
\begin{tabular}{|r||r||r|r|r|r|}
\hline
$k$	&	$ababab$	&	$aaabbb$	&	$aaaabbbb$	&	$aaaaabbbbb$	&	$aababb$	\\
\hline											
2	&	1.0000	&	1.0008	&	1.0001	&	0.9984	&	0.9996	\\
4	&	2.2500	&	2.2541	&	2.2441	&	2.2449	&	2.2502	\\
6	&	7.5000	&	7.3011	&	7.2098	&	7.2551	&	7.2319	\\
8	&	32.8125	&	30.3744	&	29.5004	&	30.0127	&	29.5378	\\
10	&	177.1880	&	155.0380	&	145.8240	&	150.7220	&	145.4910	\\
\hline											
\end{tabular}
\caption{Comparison of moments for various patterns involving $a$ and $b$. The first column are the theoretical values for the moments of the pattern ${a,b}$. The remaining columns are 200 simulations of $3600 \times 3600$ matrices. }\label{table:abpattern2}
\end{table}
\end{center}

\begin{center}
\begin{table}
\begin{tabular}{|r||r||r|r|r|r|}
\hline
$k$	&	$abcabc$	&	$abccba$	&	$aabbcc$	&	$abbcca$	&	$aabcbc$	\\
\hline											
2	&	1.0000	&	1.0005	&	1.0006	&	0.9983	&	1.0013	\\
4	&	2.1111	&	2.1122	&	2.1153	&	2.1047	&	2.1161	\\
6	&	6.1111	&	6.0248	&	6.0540	&	6.0083	&	6.0235	\\
8	&	22.0370	&	20.9398	&	21.2004	&	20.9908	&	20.8411	\\
10	&	94.6296	&	85.0241	&	87.0857	&	85.9902	&	84.2097	\\
\hline											
\end{tabular}
\caption{Comparison of moments for various patterns involving $a$ and $b$. The first column are the theoretical values for the moments of the pattern ${a,b,c}$. The remaining columns are 200 simulations of $3600 \times 3600$ matrices. }\label{table:abpattern3}
\end{table}
\end{center}

%%%%%%%%%%%%%%%%%%%%%%%%%%%%%%%%
%%%%%%%%%%%%%%%%%%%%%%%%%%%%%%%%

\subsection{The Sixth Moment}

Although for an $m$-pattern with each element appearing at a fixed frequency, the $4$\textsuperscript{th} moment is independent of how the elements are arranged within the pattern, the way the elements are arranged in an $m$-pattern does affect higher moments and thus the limiting spectral distribution. As we will show for the $6$\textsuperscript{th} moment, for patterns with repeated elements, there exist ``obstructions to modulo equations'' that make trivial some non-trivial configurations for patterns without repeated elements. We illustrate this by explicitly computing the $6$\textsuperscript{th} moment for the pattern $\{a,b,a,b\}$, and then showing why the $6$\textsuperscript{th} moment for $\{a,a,b,b\}$ differs. It will then be clear that the modulo obstructions persist for more complicated patterns and higher moments.

For the $6$\textsuperscript{th} moment, we calculate $\frac{1}{N^{\frac{6}{2}+1}}\sum_{1\leq i,j,k,l,m,n\leq N} a_{ij}a_{jk}a_{kl}a_{lm}a_{mn}a_{ni}$ by \eqref{moment}. There are $(6-1)!!=15$ matchings, which can be classified into $5$ types, so that the $6$ entries are matched in $3$ pairs:
\ben\label{6thmommat}
\item $a_{ij}=a_{jk}$, $a_{kl}=a_{lm}$, $a_{mn}=a_{ni}$ (adjacent, 2 variations).

\item $a_{ij}=a_{jk}$, $a_{kl}=a_{ni}$, $a_{lm}=a_{mn}$ (semi-adjacent-1, 3 variations).

\item $a_{ij}=a_{jk}$, $a_{kl}=a_{mn}$, $a_{lm}=a_{ni}$ (semi-adjaent-2, 6 variations).

\item $a_{ij}=a_{lm}$, $a_{jk}=a_{ni}$, $a_{kl}=a_{mn}$ (diagonal-1, 3 variations).

\item $a_{ij}=a_{lm}$, $a_{jk}=a_{mn}$, $a_{kl}=a_{ni}$ (diagonal-2, 1 variation).
\een

Similar to the $4$\textsuperscript{th} moment computation, we first take advantage of adjacent cases. For example, if $a_{ij}=a_{jk}$, then the two pairing conditions \eqref{pairconds} require
\ben
\item $i-j = k-j \Rightarrow i=k$. Given that $i,j,k,l\in\{1,2,\dots,N\}$, neither $i-j =k-j+N$ nor $i-j =k-j-N$ is possible.

\item $i\equiv k\pmod{2}$ or $j\equiv j\pmod{2}$, depending on the zone-wise location of $a_{ij}$ and $a_{jk}$. Either follows trivially from the previous condition.
\een

For Type 1 (adjacent), take $a_{ij}=a_{jk}$, $a_{kl}=a_{lm}$, $a_{mn}=a_{ni}$, the diagonal condition requires
\begin{equation} i-j = k-j, k-l =m-l, m-n =i-n \Rightarrow i=m=k. \end{equation}
By the discussion on the adjacent case, the modulo condition is satisfied trivially. We can then freely choose $i,j,l,n$, each with $N$ choices, and make a non-trivial configuration that contributes $\frac{N^4}{N^{\frac{6}{2}+1}} =1$ (fully). Type 1 matchings thus contribute $2\times 1 =2$ ($2$ variations of Type 1) to the $6$\textsuperscript{th} moment.

For Type 2 (semi-adjacent-1), take $a_{ij}=a_{jk}$, $a_{kl}=a_{ni}$, $a_{lm}=a_{mn}$, the adjacent case $a_{ij}=a_{jk}$ requires $i=k$ as discussed before. Thus the second pair $a_{kl}=a_{ni}$ is equivalent to $a_{kl}=a_{nk}$, which is again an adjacent case. The third pair $a_{lm}=a_{mn}$ is an adjacent case itself. Thus Type 2 is in fact equivalent to Type 1, and contributes $3\times 1 =3$ to the $6$\textsuperscript{th} moment.

For Type 3 (semi-adjacent-2), the adjacent case $a_{ij}=a_{jk}$ requires $i=k$ as discussed before. Thus the third pair $a_{lm}=a_{ni}$ is equivalent to $a_{lm}=a_{nk}$, and the second and the third pair combined make the diagonal matching as in the $4$\textsuperscript{th} moment computation. We have shown that this diagonal matching contributes $\frac{1}{4}$ to the $4$\textsuperscript{th} moment (see Lemma \ref{4thmom}). Note that $j$ is free with $N$ choices despite the restriction $i=k$. Thus this matching contributes $\frac{1}{4}$, and this type $6\times\frac{1}{4} =\frac{3}{2}$, to the $6$\textsuperscript{th} moment.

Note that Type 1 and 2 are independent of the $m$-block circulant pattern along the diagonals in an $m$-block circulant matrix, and Type 3 also applies to other variations of $\{a,b,a,b\}$ such as $\{a,a,b,b\}$ and $\{a,b,b,a\}$. Type 1 through 3 combined, we have $2+3+\frac{3}{2} =6.5$ in the $6$\textsuperscript{th} moment.

We proceed the diagonal matchings, for which we will discuss the modulo obstructions, and start with a simple case for Type 4. Take the matching $a_{ij}=a_{lm}$, $a_{jk}=a_{ni}$, $a_{kl}=a_{mn}$ as an example, the two pairing conditions \eqref{pairconds} require:
\ben
\item $i-j =m-l$, $j-k =i-n$, $k-l =n-m$ $\Rightarrow i-j =m-l =n-k$,\footnote{ We temporarily ignore $i-j =m-l+N$ and $i-j =m-l-N$ for simplicity. In fact, as we show in the $4$\textsuperscript{th} moment computation, the $i-j =m-l+N$ case and the $i-j =m-l+N$ case, each of which has $\frac{N^3}{6} +O(N^2)$ solutions, together make up the obstructions to a Diophantine equation like $i-j =m-l$ that has $\frac{2N^3}{3}+O(N^2)$ solutions. } which shows that we need to choose only $4$ of the $6$ indices, and the other $2$ are determined.

\item $i\mathcal{R}m$ or $j\mathcal{R}l$, $j\mathcal{R}i$ or $k\mathcal{R}n$, $k\mathcal{R}n$ or $l\mathcal{R}m$, depending on the zone-wise locations of the entries. For example, if $a_{ij}\in$ zone 1, then $i-j= m-l \Rightarrow a_{lm}\in$ zone 4. We have $2^3 =8$ sets of equivalance relations, categorized as follows.
Category (1) (4 sets): $i\mathcal{R}m\mathcal{R}j, k\mathcal{R}n$; $j\mathcal{R}l\mathcal{R}m, k\mathcal{R}n$; $i\mathcal{R}m\mathcal{R}l, k\mathcal{R}n$; $j\mathcal{R}l\mathcal{R}i, k\mathcal{R}n$.

Category (2) (2 sets): $i\mathcal{R}m\mathcal{R}j\mathcal{R}l$; $j\mathcal{R}l\mathcal{R}i\mathcal{R}m$.

Category (3) (2 sets): $i\mathcal{R}m$, $k\mathcal{R}n$; $j\mathcal{R}l$, $k\mathcal{R}n$.
\een

Each set of equivalance relations appears with a certain probability, and the probabilities of observing each $\mathcal{R}$ set sum up to $1$. We show below that, regardless of the probability of observing each set, each set contributes $\frac{1}{4}$ to the $6$\textsuperscript{th} moment, and thus the probability-weighted contribution is simply $\frac{1}{4}$.

For Cat.(1), the set of equivalence relations $i\mathcal{R}m\mathcal{R}j, k\mathcal{R}n$ requires $a_{ij}, a_{jk}, a_{kl}\in$ zone 1 or 3. Thus we can start with a free $i$ with $N$ choices, then select $m,j$, each with $\frac{N}{2}$ choices, such that $i\mathcal{R}j\mathcal{R}m$. Then we pick a $k$, and note that $i-j= n-k, i\mathcal{R}j\Rightarrow k\mathcal{R}n$. Recall that, for the pattern $\{a,b,a,b\}$, $i\mathcal{R}j$ indicates $2|(i-j)$, and it follows that $2|(n-k)$. In other words, we can freely choose a $k$ with $N$ choices, and the diagonal condition $i-j =n-k$ ensures that we have a good $n$. This set thus contributes $\frac{1}{N^4}\cdot(N\cdot\frac{N}{2}\cdot\frac{N}{2}\cdot N) =\frac{1}{4}$. The same analysis applies to the other $3$ sets in Cat.(1).

For Cat.(2), take the set $i\mathcal{R}m\mathcal{R}j\mathcal{R}l$. We start with a free $i$, and then select $m, j$, each with $\frac{N}{2}$ choices, such that $i\mathcal{R}m\mathcal{R}j$. Note that, again, $i-j=m-l, i\mathcal{R}j \Rightarrow m\mathcal{R}l \Rightarrow i\mathcal{R}m\mathcal{R}j\mathcal{R}l$. This set thus contributes $\frac{1}{N^4}\cdot(N\cdot\frac{N}{2}\cdot\frac{N}{2}\cdot N) =\frac{1}{4}$. The same analysis applies to the other set in Cat.(2).

For Cat.(3), take the set $i\mathcal{R}m, k\mathcal{R}n$. We start with a free $i$, and then select $m$ with $\frac{N}{2}$ free choices such that $i\mathcal{R}m$. Then we choose a free $k$ with $N$ choices and $n$ with $\frac{N}{2}$ choices such that $k\mathcal{R}n$. This set thus contributes $\frac{1}{N^4}\cdot(N\cdot\frac{N}{2}\cdot N\cdot\frac{N}{2}) =\frac{1}{4}$. The same analysis applies to the other set of Cat.(3).

Since each individual set of equivalence relations in Cat.(1)-(3) contributes equally, the probability-weighted contribution to the $6$\textsuperscript{th} moment is $\frac{1}{4}$. Therefore, Type 4, with $3$ variations, contributes $\frac{3}{4}$.

Similarly, the pairing conditions \eqref{pairconds} entail the following for Type 5 (diagoal 2).
\ben
\item $i-j =m-l =k-n$.

\item 2 categories of equivalence relation set.

Cat.(1)(6 sets): $i\mathcal{R}m\mathcal{R}k, j\mathcal{R}n$; $j\mathcal{R}l\mathcal{R}n, k\mathcal{R}m$; $i\mathcal{R}m, j\mathcal{R}n\mathcal{R}l$; $j\mathcal{R}l, k\mathcal{R}m\mathcal{R}i$; $i\mathcal{R}m\mathcal{R}k, l\mathcal{R}n$; $j\mathcal{R}l\mathcal{R}n, k\mathcal{R}i$.

Cat.(2)(2 sets): $i\mathcal{R}m\mathcal{R}k$; $j\mathcal{R}l\mathcal{R}n$.
\een
Replicating the analysis of Type 4, we find that, since each set of equivalence relations in Cat.(1) and Cat.(2) contributes $\frac{1}{4}$, the probability-weighted contribution must be $\frac{1}{4}$ as well. Since Type 5 has only $1$ variation, Type 5 contributes $\frac{1}{4}$ to the $6$\textsuperscript{th} moment.

Therefore, the combined contribution from Type 4 and Type 5 is $\frac{3}{4}+\frac{1}{4}=1$. The $6$\textsuperscript{th} moment for the pattern $\{a,b,a,b\}$ is then $6.5+1 =7.5$.

Now we examine why the contribution from diagonal matchings for the pattern $\{a,a,b,b\}$ differs from that for $\{a,b,a,b\}$. As discussed before, Type 1 through 3 matchings, with a total contribution of $6.5$, also apply to $\{a,a,b,b\}$. For Type 4 and 5, however, the combined contribution is less than $1$. Recall a key argument in the analysis of Type 4 matching before: for the $2$-block circulant $\{a,b,a,b\}$ pattern, under $i-j=m-l=n-k$, if we choose $i\mathcal{R}j\mathcal{R} m$, i.e. $i\con j\con m\Mod{2}$, then $l=j+m-i$ will satisfy $l\con m\Mod{2}$ as well. Namely, $i-j=m-l, i\mathcal{R}j \Rightarrow m\mathcal{R}l$. However, for $\{a,a,b,b\}$, if we specify $\mathcal{R}$ as $s\mathcal{R}t\equi\mod{(s,4)},\mod{(t,4)}\in\{1,2\}$ or $\mod{(s,4)},\mod{(t,4)}\in\{0,3\}$, and choose $i\mathcal{R}j\mathcal{R}m$, it is possible that $l=j+m-i$ is not related to $m$ under $\mathcal{R}$. For instance, when $\mod{(i,4)}=2,\mod{(j,4)}=1,\mod{(m,4)}=3$, we have $i\mathcal{R}j$, but $\mod{(l,4)}=2$. Some configurations that are non-trivial for $\{a,b,a,b\}$ then become trivial for $\{a,a,b,b\}$, while all the non-trivial configurations for $\{a,a,b,b\}$ are still non-trivial for $\{a,b,a,b\}$. Thus, we expect the $6$\textsuperscript{th} moment for $\{a,a,b,b\}$ to be smaller than that for $\{a,b,a,b\}$, which is also evidenced by numerics. We phrase such a loss of non-trivial configurations as due to ``obstructions to modulo equations'', or ``modulo obstructions'' for short, which will clearly persist in higher moments for general $m$-block circulant patterns with repeated elements.

Based on the brute-force computation above, we may also bound the even moments for generalized $m$-block circulant patterns. It is clear that a lower bound is the moment for the $m$-block circulant pattern of the same period length and in which each element is distinct. For example, in terms of high ($2k$\textsuperscript{th}, $k\ge 2$) moments, $\{a,a,b,b\}>\{a,b,c,d\}$ (both of length $4$). In the computation of high moments, a pattern with repeated elements has all the non-trivial configurations that an all-distinct pattern of the same length can have, and gains extra non-trivial configurations due to the repeated elements.

An easy upper bound is the moment of the standard Gaussian, which is the limiting spectral distribution for the ensemble of circulant matrices. We may also easily find a sharper upper bound for a family of simple $m$-block circulant patterns in which each element appears at the same frequency, e.g. $\{a,b,c,c,b,a\}$, $\{a,a,b,c,b,c\}$, etc. For this family, an upper bound will be associated with a pattern where each element only appears once. For example, in terms of high moments, $\{a,b,c\}>\{a,b,c,c,b,a\}$. We may take $\{a,b,c\}$ as $\{a,b,c,a,b,c\}$, and note that, although in $\{a,b,c,a,b,c\}$, the probability of choosing each letter is the same as in $\{a,b,c,c,b,a\}$, the former pattern is free of modulo obstructions that exist for the latter.

For a more general $m$-block circulant pattern, however, a sharper upper bound is not easily attainable. For instance, it is not clear whether $\{a,a,b,c\}>\{a,b,c\}$. Some numeric evidence suggests that a pattern in which $gcd(\nu_1,\nu_2\dots \nu_\ell)=1$, where $\nu_\ell$ is the number of occurrences of an element in an $m$-period, has larger high moments than those with the same frequency of each element but $gcd(\nu_1,\nu_2\dots \nu_\ell)\ge 2$. For example, $\{a,b,c,c\}>\{a,a,b,b,c,c,c,c\}$.

Obviously, the accounting above will become significantly more involved for more complicated patterns or higher moments, but the basic ideas remain the same. We also foresee that as the moments get higher, the number of configurations that contribute trivially will increase so quickly that the higher moments get increasingly farther below the standard Gaussian moments. This is also evidenced by simulations.

%%%%%%%%%%%%%%%%%%%%%%%%%%%%%%%%%%%%%%%%%%%%%%%%%%%%%%%%%%%%%%%%
%%%%%%%%%%%%%%%%%%%%%%%%%%%%%%%%%%%%%%%%%%%%%%%%%%%%%%%%%%%%%%%%

\subsection{Existence and Convergence of High Moments}

Although it is impractical to find every moment for a general $m$-block circulant pattern using brute-force computation, we are still able to prove that, for any $m$-block circulant pattern, every moment exists, is finite (and satisfies certain bounds), and that there exists a limiting spectral distribution. In addition, the empirical spectral measure of a typical real symmetric $m$-block circulant matrix converge to this limiting measure, and we have convergence in probability and almost sure convergence.

We have shown that all the odd moments vanish as $N\to\infty$, and thus we focus on the even moments. We need to prove the following theorem.

\begin{thm} For any patterned $m$-block circulant matrix ensemble, $\lim_{N\to\infty} M_{2k}(N)$ exists and is finite. \end{thm}

\begin{proof}
It is trivial that $M_{2k}(N)$ is finite. As discussed before, it is bounded below by the $2k$\textsuperscript{th} moment for the ensemble of $m$-block circulant matrices where, in the $m$-pattern, each element is distinct, and more importantly it is bounded above by the $2k$\textsuperscript{th} moment for the ensemble of circulant matrices, and we know that the limiting spectral distribution for this matrix ensemble is a Gaussian.

We now show that $\lim_{N\to\infty} M_{2k}(N)$ exists. To calculate $M_{2k}(N)$, we match $2k$ elements from the matrix, $\{a_{i_1i_2},a_{i_2i_3},\dots ,a_{i_{2k}i_1}\}$, in $k$ pairs, where $i_\ell\in\{1,2,\dots ,N\}$ and this will give $(2k-1)!!$ matchings. For each matching, there are a certain number of configurations, and most of such configurations do not contribute to the moments as $N\to\infty$.

For the $m$-block circulant pattern, the equivalence relation $\mathcal{R}$ implies that $\epsilon_{s}\mathcal{R}\epsilon_{t+1} \Leftrightarrow\epsilon_{s}=\epsilon_{t+1}$, and since $m|(\epsilon_{s}-\epsilon_{s+1}+\epsilon_{t}-\epsilon_{t+1})$, we have $\epsilon_{s+1}=\epsilon_{t}$ as well (see \eqref{circpairconds}).\footnote{ This explains why, for an $m$-pattern without repeated elements, the zone-wise locations of matrix entries do not matter in making a non-trivial configuration.} Thus $\eta_{s}-\eta_{s+1}+\eta_{t}-\eta_{t+1} =0\text{ or }\pm\frac{N}{m}$, three equations that have $(\frac{N}{m})^3 +O((\frac{N}{m})^2)$ solutions in total, as we have shown in the $4$\textsuperscript{th} moment computation.

However, if there are repeated elements in an $m$-period, then $\epsilon_{s}\mathcal{R}\epsilon_{t+1}$ no longer necessitates $\epsilon_{s}=\epsilon_{t+1}$, and it is possible that $(\epsilon_{s}-\epsilon_{s+1}+\epsilon_{t}-\epsilon_{t+1})=\pm m$. Thus, the zone-wise locations of elements matter in making non-trivial configurations. Recall that the zone-wise location (see \eqref{zones}) of an element $a_{i_{s}i_{s+1}}$ is determined by $(i_{s+1}-i_{s})$: if $a_{i_{s}i_{s+1}}$ is in zone 1 or 3 (Area I), $\epsilon_s$ determines the slot of $a_{i_{s}i_{s+1}}$ in an $m$-period; if $a_{i_{s}i_{s+1}}$ is in zone 2 or 4 (Area II), $\epsilon_{s+1}$ determines the slot of $a_{i_{s}i_{s+1}}$ in an $m$-period. In addition, the diagonal condition will always ensure that two paired entries $a_{i_{s}i_{s+1}}$ and $a_{i_{t}i_{t+1}}$ are located in different areas.
% For instance, if $a_{i_{s}i_{s+1}}\in$ zone 1 and $i_{s}-i_{s+1}=-(i_{t}-i_{t+1})$, then $a_{i_{s}i_{s+1}}\in$ zone 4; if $a_{i_{s}i_{s+1}}\in$ zone 1 and $i_{s}-i_{s+1} =-(i_{t}-i_{t+1})-N$, then $a_{i_{s}i_{s+1}}\in$ zone 2, etc.

Recall that for any matching $\mathcal{M}$, the $k$ pairs of matrix elements, each pair in the form of $a_{i_{s}i_{s+1}}=a_{i_{t}i_{t+1}}$, are fixed. For any $\mathcal{M}$, to make a non-trivial configuration, we first choose an $\epsilon$ vector of length $2k$. If we choose all the $\epsilon_\ell$'s freely, there are $m^{2k}$ possible choices for an $\epsilon$ vector, most of which do not meet the modulo condition, and trivially, $m^{2k}$ is an upper bound for the number of valid $\epsilon$ vectors. It is noteworthy that out of the $2k$ $\epsilon_\ell$'s of an $\epsilon$ vector, only some of the $\epsilon_\ell$'s will matter for the modulo condition. Which $\epsilon_\ell$'s in fact matter depends on how we pair the $2k$ matrix entries $a_{i_{s}i_{s+1}}$'s and the zone-wise locations of the paired $a_{i_{s}i_{s+1}}$'s, which we cannot determine without fixing the $\eta_\ell$'s (and thus the $i_\ell$'s).

However, for any matching, the way we pair the $2k$ matrix entries into $k$ pairs is fixed, and for each fixed pair $a_{i_{s}i_{s+1}}=a_{i_{t}i_{t+1}}$, two $\epsilon_\ell$'s will matter for the modulo condition: either $\epsilon_{s}\mathcal{R}\epsilon_{t+1}$ or $\epsilon_{s+1}\mathcal{R}\epsilon_{t}$. Thus there are $2^k$ ways to choose $k$ pairs of $\epsilon_\ell$'s for each matching. For each way of fixing the $k$ pairs of $\gep_\ell$'s, we examine each $\gep$ pair, say $(\gep_{\ell_1},\gep_{\ell_2})$, and there are a certain number of choices of $(\gep_{\ell_1},\gep_{\ell_2})$ such that $\gep_{\ell_1}\mathcal{R}\gep_{\ell_2}$. Continuing in this way, for each $\epsilon$ pair, we choose two $\epsilon_\ell$'s that satisfy the equivalence relation $\mathcal{R}$. Note that an $\gep_\ell$ may matter twice, once, or never for the modulo condition depending on the zone-wise locations of the $a_{i_{s}i_{s+1}}$'s. We then choose the other $\gep_\ell$'s that do not matter for the modulo condition such that for each pair of $a_{i_{s}i_{s+1}}=a_{i_{t}i_{t+1}}$, we have $\epsilon_{s}-\epsilon_{s+1}+\epsilon_{t}-\epsilon_{t+1} =0\text{ or }\pm m$, and finally we have a valid $\epsilon$ vector. The number of valid $\epsilon$ vectors will be determined by $m$, $k$, and the pattern of an $m$-period, but will be independent of $N$ since the system of $k$ equivalence relations for the modulo condition does not involve $N$.

With a valid $\epsilon$ vector, we have fixed the zone-wise locations of the $2k$ matrix elements by fixing the $\epsilon_\ell$'s that matter for the modulo condition. We now turn to the diagonal condition and study the $\eta_\ell$'s. With $k$ equations in the form of
\be m(\eta_{s}-\eta_{s+1}+\eta_{t}-\eta_{t+1})+(\epsilon_{s}-\epsilon_{s+1}+\epsilon_{t}-\epsilon_{t+1}) =0\text{ or }\pm N, \ee
and $(\epsilon_{s}-\epsilon_{s+1}+\epsilon_{t}-\epsilon_{t+1})$ known in each of the $k$ equations, we in fact have $k$ equations in the form of
\be \eta_{s}-\eta_{s+1}+\eta_{t}-\eta_{t+1}=\gamma,\ee
where $\gamma\in\{0,\pm 1,\frac{N}{m},\frac{N}{m}\pm 1,-\frac{N}{m},-\frac{N}{m}\pm 1\}$. This gives us $k+1$ degrees of freedom in choosing the $\eta_\ell$'s, and trivially, we can have at most $(\frac{N}{m})^{k+1}$ vectors of $\eta_\ell$'s. Since the $\epsilon$ vector is fixed, for one equation $\eta_{s}-\eta_{s+1}+\eta_{t}-\eta_{t+1}=\gamma$, there are only $3$ choices of $\gamma$. With $k$ equations in this form, we have at most $3^k$ systems of $\eta$ equations. Note that not all of the $\eta$ vectors satisfying an $\eta$ equation system derived from the diagonal condition will help make a non-trivial configuration, since the $\eta_\ell$'s need to be chosen such that the resulted $a_{i_{s}i_{s+1}}$'s will satisfy the zone-wise locations in order to be coherent with the pre-determined $\epsilon$ vector. For example, if in a pair of matrix entries $a_{i_{s}i_{s+1}}=a_{i_{t}i_{t+1}}$ where $\epsilon_{s}\mathcal{R}\epsilon_{t+1}$, even though the $\eta_\ell$'s are chosen such that $\eta_{s}-\eta_{s+1}+\eta_{t}-\eta_{t+1}=\gamma$, it is possible that $a_{i_{s}i_{s+1}},a_{i_{t}i_{t+1}}$ are located in certain zones such that we need $\epsilon_{s+1}\mathcal{R}\epsilon_{t}$ to ensure a non-trivial configuration.

The following steps mirror those in \cite{HM}. Denote an $\eta$ equation system by $\mathcal{S}$. For any $\mathcal{S}$ we have $k$ equations with $\eta_1, \eta_2,\dots,\eta_{2k}\in\{1,2,\dots,\frac{N}{m}\}$. Let $z_\ell=\frac{\eta_\ell}{N/m}\in\{\frac{m}{N},\frac{2m}{N},\dots,1\}$. Without the zone-wise concerns discussed before, the system of $k$ equations would have $k+1$ degrees of freedom and determine a nice region in the $(k+1)$-dimensional unit cube. Taking into account the zone-wise concerns, however, we will still have $k+1$ degrees of freedom. For example, for a pair of matrix elements $a_{i_{s}i_{s+1}}=a_{i_{t}i_{t+1}}$, the system $\mathcal{S}$ requires  $\eta_{s}-\eta_{s+1}+\eta_{t}-\eta_{t+1}=\gamma$. If we need $\epsilon_{s}\mathcal{R}\epsilon_{t+1}$ to make a non-trivial configuration, say $a_{i_{s}i_{s+1}}\in$ zone 1, then we will obtain an additional equation $0\leq i_{s+1}-i_{s}\leq\frac{N}{2}-1 \Rightarrow 0\leq (\eta_{s+1}-\eta_{s})+\epsilon_{s+1}-\epsilon_{s}\leq\frac{N}{2}-1$ with $(\epsilon_{s+1}-\epsilon_{s})\in\{-m+1,-m+2,\dots,0,1,\dots, m-2,m-1\}$. Based on the region determined by $\eta_{s}-\eta_{s+1}+\eta_{t}-\eta_{t+1}=\gamma$, this additional zone-related restriction will only allow a slice of the region for us to choose valid $\eta_\ell$'s. With $k$ zone-wise restrictions, only a proportion of the original region in the unit cube will be preserved for the choice of the $\eta$ vector. Nevertheless, the ``width'' of each slice is of order $\frac{N}{2}$, and we still have $k+1$ degrees of freedom.

Therefore, with $m$ fixed and as $N\to\infty$, we obtain to first order the volume of this region, which is finite. Unfolding back to the $\eta_\ell$'s, we obtain $M_{2k}(\mathcal{S})(\frac{N}{m})^{k+1}+O_k((\frac{N}{m})^k)$, where $M_{2k}(\mathcal{S})$ is the volume associated with this $\eta$ system. Summing over all $\eta$ systems, we obtain the number of non-trivial configurations for the $2k$\textsuperscript{th} moment from this particular $\epsilon$ vector. Next, within a given matching $\mathcal{M}$, we sum over all valid $\epsilon$ vectors, the number of which is independent of $N$ as we have shown before. In the end, we sum over the $(2k-1)!!$ matchings to obtain $M_{2k}N^{k+1}+O_k(N^k)$, and the $2k$\textsuperscript{th} moment is simply $\frac{M_{2k}N^{k+1}+O_k(N^k)}{N^{k+1}}=M_{2k}+O(\frac{1}{N})$.
\end{proof}

The above proves the existence of the moments. The convergence proof follows with only minor changes to the convergence proofs from \cite{HM,MMS}.

%%%%%%%%%%%%%%%%%%%%%%%%%%%%%%%%%%%%%%%%%%%%%%%%%%%%%%%%%%%%%%%%
%%%%%%%%%%%%%%%%%%%%%%%%%%%%%%%%%%%%%%%%%%%%%%%%%%%%%%%%%%%%%%%%

\ \\

\end{document}